\documentclass[12pt,reqno]{amsart}

\usepackage[colorlinks=true, pdfstartview=FitV, linkcolor=blue, citecolor=blue, urlcolor=blue]{hyperref}
\usepackage{amssymb,amsmath,amsthm}
\usepackage{comment}
\usepackage{amscd}
\usepackage{bbm}
\usepackage{times, verbatim}
\usepackage{graphicx}
\usepackage[english]{babel}
 \usepackage[usenames, dvipsnames]{color}
\usepackage{enumerate}
\usepackage{anysize}
\marginsize{3cm}{3cm}{3cm}{3cm}
\input xy
\xyoption{all}
\usepackage{pb-diagram}
\usepackage[all]{xy}
\input xy
\xyoption{all}

\DeclareFontFamily{OT1}{rsfs}{}
\DeclareFontShape{OT1}{rsfs}{n}{it}{<-> rsfs10}{}
\DeclareMathAlphabet{\mathscr}{OT1}{rsfs}{n}{it}

\DeclareMathOperator{\Ext}{Ext}

\DeclareMathOperator{\Hom}{Hom}
\DeclareMathOperator{\Rep}{Rep}
\DeclareMathOperator{\GL}{GL}

\DeclareMathOperator{\SL}{SL}
\DeclareMathOperator{\G}{G}

\DeclareMathOperator{\St}{St}

\DeclareMathOperator{\cind}{c-ind}
\DeclareMathOperator{\Ind}{Ind}
\DeclareMathOperator{\csupp}{csupp}

\DeclareMathOperator{\Irr}{Irr}

\DeclareMathOperator{\cosoc}{cosoc}

\newcommand{\csuppline}{{\csupp_{\mathbb{Z}}}}

\newcommand{\A}{{\mathcal{A}}}

\newtheorem{theorem}{Theorem}[section]
\newtheorem{proposition}{Proposition}[section]
\newtheorem{corollary}{Corollary}[section]
\newtheorem{example}{Example}[section]

\theoremstyle{definition}
\newtheorem{remark}{Remark}[section]

\theoremstyle{definition}

\newtheorem{lemma}{Lemma}[section]
\newtheorem{definition}{Definition}[section]

\theoremstyle{definition}

\theoremstyle{definition}

\numberwithin{equation}{section}

\begin{document}

\title[Ext embedding theorem for $p$-adic GL]{An inductive Ext non-vanishing theorem for the $p$-adic general linear group}

\author{Kei Yuen Chan and Mohammed Saad Qadri}
\date{}

\maketitle

\begin{abstract}
We study some homological properties of the parabolic induction functor for the $p$-adic general linear group. We obtain an embedding theorem of Ext-groups in the context of parabolic induction. As an application, we establish and prove a variation of the non-tempered Gan-Gross-Prasad conjecture in homological branching laws for $p$-adic general linear groups.
\end{abstract}

\section{Introduction}

Let $F$ be a non-archimedean local field. The category of smooth representations of the $p$-adic general linear group is not a semi-simple category and so one is naturally interested in the study of homological aspects of the representation theory of $\GL_n(F)$. The two main goals of this article are:
\begin{enumerate}
    \item To study the relationship between $\Ext$ modules and the parabolic induction functor (see Theorem \ref{theorem on injective mapping on ext modules}). 
    \item To study some homological aspects of the restriction problem for the $p$-adic general linear group as initiated in \cite{pr18} (see Theorem \ref{theorem on non tempered ext branching}). 
\end{enumerate} As we shall see a study of the former shall shed some light on the latter question. In what follows, we refer to the body of the paper for undefined notations and terminology.

\subsection{Extensions and Parabolic Induction}

The parabolic induction functor plays a central role in the representation theory of reductive $p$-adic groups. It is an important tool in the construction of representations of $\GL_n(F)$. Some questions such as the study of irreducibility under parabolic induction and the study of composition factors of parabolically induced modules have attracted much attention. On the other hand, the behavior of extensions under parabolic induction seems to be less well studied. In this article, we study some homological properties of parabolic induction functor. We first give some motivation for our study.

As the following example \cite{ze80} illustrates, it is known that some nontrivial extensions may be trivialized under parabolic inductions and so it is interesting to understand the behavior of extensions under parabolic induction. 

\begin{example}
Consider the following short exact sequence of representations of $\GL_2(F)$ which is non-split: \[ 0 \rightarrow \mathbbm{1}_2 \rightarrow \nu^{-1/2}\times \nu^{1/2} \rightarrow \St_2 \rightarrow 0. \] However, the short exact sequence, \[ 0 \rightarrow \mathbbm{1}_2 \times \nu^{-1/2} \rightarrow \nu^{-1/2}\times \nu^{1/2} \times \nu^{-1/2} \rightarrow \St_2 \times \nu^{-1/2} \rightarrow 0, \] of representations of $\GL_3(F)$ splits, that is, $\nu^{-1/2}\times \nu^{1/2} \times \nu^{-1/2}$ is decomposable.
\end{example}

The above example also illustrates that one must impose some conditions on cuspidal support in order to hope that extensions are preserved under the parabolic induction functor.

Another motivation for the study of the behavior of extensions under parabolic induction comes from the work \cite{ch22} of the first author on the proof of the local non-tempererd Gan-Gross-Prasad conjecture for $\GL_n(F)$. The following result plays a crucial role in the proof of the non-tempered GGP conjecture for the $p$-adic general linear group by the first author.

Let $\omega \in \Rep(\GL_k(F))$ be a Speh representation built out of a discrete series. Let $\mathcal C_{\omega}$ denote the full subcategory of $\Rep(\GL_m(F))$ of finite-length representations $\sigma$ and such that for any composition factor $\sigma^{\prime}$ of $\sigma$ and for any $\rho \in \csupp(\sigma^{\prime})$, either: \begin{enumerate}
    \item $\rho \in \csupp(\omega)$.
    \item $\rho \not \in (\csupp_\mathbb{Z}(\omega) - \csupp(\omega))$.
\end{enumerate}

\begin{theorem}\cite[Theorem 8.1]{ch22}\cite{ch24product} \label{theorem on fully faithful}
Let $\pi\in \mathcal C_{\omega}$ be a smooth representation of $\GL_m(F)$ with length equal to 2. Then $\pi$ is indecomposable if and only if $\omega\times \pi$ is indecomposable.
\end{theorem}

The above result for first $\Ext$ modules suffices for the purpose of understanding the $\Hom$ branching law. Since we are interested in the study of $\Ext$ branching laws, one has to study the behavior of the higher $\Ext$ modules under parabolic induction. We obtain the following result on the behavior of $\Ext$ modules under parabolic induction, which is the first main result of this article.

\begin{theorem} \label{theorem on injective mapping on ext modules}
Let $i$ be a non-negative integer. Let $\tau_1$ and $\tau_2$ belong to the category $\mathcal C_\omega$. Then there is an injective map $\phi$,
 \[ \phi: \mathrm{Ext}_{\GL_m(F)}^i(\tau_1, \tau_2)\hookrightarrow \mathrm{Ext}_{\GL_{m+k}(F)}^i(\omega \times \tau_1, \omega \times \tau_2). \] 
\end{theorem}

As a corollary of the above result, we see that if $\tau_1$ and $\tau_2$ belong to the category $\mathcal C_\omega$ then for $i \in \mathbb{Z}_{>0}$, $\mathrm{Ext}^i(\tau_1, \tau_2)\neq 0$ implies  $\mathrm{Ext}^i(\omega \times \tau_1, \omega \times \tau_2)\neq 0$, which justifies the term ``inductive" in the title. As noted earlier, it is natural to impose some conditions on the cuspidal support in order to hope that extensions are preserved under the parabolic induction functor. As producting with Speh representations is essential in the construction of Arthur type representations and unitary representations, this results should be useful in the study of Ext-groups of those representations. Indeed, we shall resolve a conjecture in \cite{qa25} about Arthur type representations (see Theorem \ref{theorem on non tempered ext branching} below).

\begin{remark}
One cannot expect that the map $\phi$ in the above theorem to be surjective in general. For instance, from the projective resolution for the trivial representation provided by the Bruhat-Tits building associated to $\GL_n(F)$ we know that, \[ \Ext^6_{\GL_5(F)}(\mathbbm{1}_5,\St_5) = 0. \] On the other hand, one can check that, \[ \Ext^6_{\GL_{10}(F)}(\mathbbm{1}_5 \times \mathbbm{1}_5,\mathbbm{1}_5\times \St_5) \neq 0. \]
\end{remark}

\begin{remark}
Let $\mathcal A$ be the abelian category containing precisely three simple non-zero vector spaces over $\mathbb C$, say $X, Y$ and $Z$. Suppose that the only non-semisimple objects in $\mathcal A$ are $P$ and $Q$, which admits the following short exact sequences respectively:
\[   0 \rightarrow Y \rightarrow P \rightarrow X \rightarrow 0 
\]
and
\[  0 \rightarrow Z \rightarrow Q \rightarrow Y \rightarrow 0 .
\]
Moreover, 
\[ \mathrm{Hom}_{\mathcal A}(X, X), \mathrm{Hom}_{\mathcal A}(Y, Y), \mathrm{Hom}_{\mathcal A}(Z, Z), \mathrm{Hom}_{\mathcal A}(P, P), \mathrm{Hom}_{\mathcal A}(Q, Q) 
\]
contains precisely only the scalar multiple maps. Note that $P$ and $Q$ are projective objects in $\mathcal A$. Let $\mathcal B$ be a category containing $\mathcal A$ as its full subcategory, and containing precisely one more object $R$ satisfying the following properties:
\begin{enumerate}
\item of length $3$;
\item indecomposable;
\item there is a monic from $Q$ to $R$;
\item there is an epic from $R$ to $P$.
\end{enumerate}
 We observe that $\mathrm{Ext}_{\mathcal A}^2(X, Z) \cong \mathbb C$ by applying the following projective resolution on $X$ in $\mathcal A$:
\[    0 \rightarrow   Z\rightarrow  Q \rightarrow  P  \rightarrow X \rightarrow 0
\]
On the other hand, $\mathrm{Ext}_{\mathcal B}^2(X, Z)=0$ by applying the following projective resolution in $\mathcal B$:
\[      0\rightarrow Q  \rightarrow R \rightarrow X \rightarrow 0 .
\]
This results in $\mathrm{Ext}_{\mathcal B}^2(X, Z)=0$. Thus, the natural embedding from $\mathcal A$ to $\mathcal B$ is a fully-faithful exact functor, but the induced map from $\mathrm{Ext}^2_{\mathcal A}(X,Z)$ to $\mathrm{Ext}^2_{\mathcal B}(X,Z)$ is not faithful. This example illustrates why Theorem \ref{theorem on injective mapping on ext modules}
cannot be directly inferred from Theorem \ref{theorem on fully faithful} or \cite[Theorem 9.1]{ch22}.

\end{remark}

\begin{remark}
We point out that the cuspidal support conditions appearing in the definition of the category $\mathcal C_{\omega}$ in Theorem \ref{theorem on injective mapping on ext modules} are motivated by the analysis of when the product of a cuspidal representation with $\omega$ is irreducible.     
\end{remark}

\begin{remark}
Since the category of smooth representations of $\GL_n(F)$ is not a semi-simple category, it is an interesting question to understand and describe the indecomposable representations (of length $\geq 2$). One may expect that Theorem \ref{theorem on injective mapping on ext modules} may prove useful in this regard.
\end{remark}

\begin{remark}
The category occurring in the statement of the fully faithful functor result in \cite[Theorem 1.2]{ch24product} is more general than the subcategory $\mathcal C_{\omega}$ defined above. Using the results of \cite{ch24product} it may be possible to generalize Theorem \ref{theorem on injective mapping on ext modules} to the case where $\omega \in \Rep(\GL_k(F))$ is any irreducible representation.
\end{remark}

\subsection{Homological Branching Law}
Another purpose of this article is to investigate some homological aspects of the restriction problem for the $p$-adic general linear group. We first review some of the progress in the study of the branching laws for the $p$-adic general linear group. Let $\pi$ and $\pi^{\prime}$ denote irreducible representations of $\GL_n(F)$ and $\GL_{n-1}(F)$ respectively.
\begin{enumerate}
\item The multiplicity one theorem for the space $\Hom_{\GL_{n-1}(F)}(\pi,\pi^{\prime})$ is established in \cite{agrs10} (see also \cite{ch23}).
\item The $\Hom$ branching law for generic representations is obtained in \cite{jpss83,pr93}.
\item The notion of $\Ext$ analogues of branching laws is proposed in the work \cite{pr18}. 
\item The $\Ext$ branching law for generic representations conjectured in \cite{pr18} is established in the work \cite{cs21} of Chan-Savin (see also \cite{cr23},\cite{wz25}).
\item In a recent paper \cite{ggp20}, Wee Teck Gan, Benedict Gross, and Dipendra Prasad introduce the non-tempered GGP conjectures for Arthur type representations for certain pairs of $p$-adic classical groups.
\item The non-tempered GGP conjecture for the $p$-adic general linear group is fully resolved by the first author in \cite{ch22}. 
\item The $\Hom$ branching law for arbitrary irreducible representations is determined in \cite{ch23generalbranchinglaw} (see also \cite{pa25}).
\end{enumerate}

The next main result of this article is motivated by the non-tempered Gan-Gross-Prasad conjecture for the $p$-adic general linear group. It is noted in \cite[Remark 5.8]{ggp20}, that is an intriguing question to identify the precise conditions under which $\Ext^i_{\GL_{n-1}(F)}(\pi,\pi^{\prime})$ is non-zero for some $i\geq 0$, when $\pi$ and $\pi^{\prime}$ are representations of Arthur type. This problem is not completely understood, with partial progress obtained in \cite{ch22, qa25}. In this paper, we build upon these previous results to provide sufficient conditions for $\Ext$ branching between Arthur-type representations, which is the second main result of the paper. 

Our $\Ext$ non-vanishing result is based on the notion of strong $\Ext$ relevance for a pair of Arthur type representations introduced in \cite{qa25}. We defer the precise definition of this notion to Section \ref{section -  definition of strong ext relevance} and only point out that the definition of strong $\Ext$ relevance is motivated by the original definition of GGP relevance in \cite{ggp20} and a duality theorem of Nori and Prasad \cite{np20} involving the Aubert-Zelevinsky involution. The second main result of this paper is the following result \cite[Conjecture 1.3]{qa25}.

\begin{theorem} \label{theorem on non tempered ext branching}
Let $\pi$ and $\pi^{\prime}$ be Arthur type representations of $\GL_n(F)$ and $\GL_{n-1}(F)$ respectively. If $\pi$ and $\pi^{\prime}$ are strong $\Ext$ relevant then, 
\[ \Ext^i_{\GL_{n-1}(F)}(\pi,\pi^{\prime})\neq 0 \] for some integer $i\geq 0$.
\end{theorem}

\begin{remark} We note that, in general, strong $\Ext$ relevance does not give as a necessary condition for $\Ext$ non-vanishing (see Page 6 of \cite{qa25}). Hence we only obtain one direction of the $\Ext$ branching law. Therefore the above result makes only partial progress towards an $\Ext$ analogue of the local non-tempered Gan-Gross-Prasad conjecture.
 
This work is written in the hope that it may be possible to identify a general condition controlling the $\Ext$ branching law for Arthur type representations, allowing one to obtain a complete $\Ext$ analogue of the local non-tempered Gan-Gross-Prasad conjecture for the $p$-adic general linear group (see also a recent related work \cite{clltz25}). \end{remark}

\begin{remark}
For the proof of the above result, it is crucial to study the relationship between extensions and parabolic induction and hence Theorem \ref{theorem on injective mapping on ext modules} shall play a crucial role. 

Moreover, given an irreducible representation $\pi$ of $\GL_n(F)$ its restriction to $\GL_{n-1}(F)$ is typically of infinite length. However, one knows that $\pi|_{\GL_{n-1}(F)}$ is locally finitely generated, that is, each Bernstein component of $\pi|_{\GL_{n-1}(F)}$ is finitely generated \cite{as20, {cs21}}. Hence in order to establish the $\Ext$ branching law stated above we need to deal with finitely generated representations. Here, a key input that we shall need will be \cite[Proposition 5.2]{np20}.
\end{remark}

\subsection{Organization of the Paper}
We describe the structure of the paper. In Section \ref{section - Bernstein decomposition}, we recall some material on the Bernstein decomposition and affine hecke algebras of type A. We recall the various classes of irreducible representations of $\GL_n(F)$ and the notion of Bernstein-Zelevinsky derivatives in Section \ref{section - classification and bz derivatives}.

We extend a result on the full faithfulness of the parabolic induction functor from \cite{ch22} to finitely generated representations in Section \ref{section - full faithfulness on a completed category}. This result plays a crucial role in the proof of Theorem \ref{theorem on injective mapping on ext modules} and hence Theorem \ref{theorem on non tempered ext branching}. We then prove Theorem \ref{theorem on injective mapping on ext modules} in Section \ref{section - on extensions and parabolic induction}. We describe the notion of strong $\Ext$ relevance in Section \ref{section -  definition of strong ext relevance} and prove Theorem \ref{theorem on non tempered ext branching} in Section \ref{section - theorem on ext branching}. 

\subsection{Acknowledgements} 
Some of the results were announced in Tianyuan Conference in August 2025. The first-named author would like to thank the organizers for their kind invitations. This project is supported in part by the Research Grants Council of the Hong Kong Special Administrative Region, China (Project No: 17305223, 17308324) and the National Natural Science Foundation of China (Project No. 12322120). 

\section{Bernstein Decomposition and Affine Hecke Algebras of Type A} \label{section - Bernstein decomposition}

We first recall some preliminary material and set up notations. \subsection{Basic Notations and Conventions}

\begin{itemize}
    \item We fix a non-archimedean local field $F$ and set $\G_n = \GL_n(F)$ for each positive integer $n$.
    \item Throughout this work, we only consider smooth representations of $p$-adic groups over $\mathbb{C}$.
    \item We denote by $\Rep(\G_n)$ the category of smooth representations of $\G_n$, and set $\Rep = \bigoplus_{n\geq 0} \Rep(\G_n)$.
    \item The set of irreducible representations of $\G_n$ is denoted by $\Irr(\G_n)$, and we define $\Irr = \bigcup_{n\geq 0} \Irr(\G_n)$.
    \item For any non-negative integer $k$ and any representation $\pi \in \Rep(\G_k)$, we denote $n(\pi) = k$.
    \item We write $\nu$ for the character of $\G_n$ given by $\nu(g) = |\det(g)|_F$, where $g \in \G_n$.
\end{itemize}

\subsection{Parabolic Induction} Given any two representations
$\pi_1 \in \Rep(\G_{n_1}(F))$ and $\pi_2 \in \Rep(\G_{n_2}(F))$, where
$n_1, n_2 \in \mathbb{Z}_{\geq 0}$, we define the parabolically induced representation $\pi_1\times \pi_2$ as follows. The 
standard parabolic subgroup $P_{n_1,n_2}$ of $\G_{n_1+n_2}(F)$ corresponding to 
the partition $(n_1,n_2)$ has the
Levi decomposition $P_{n_1,n_2} = M_{n_1,n_2}N_{n_1,n_2}$ with 
$M_{n_1,n_2} \cong \G_{n_1}(F) \times \G_{n_2}(F)$ . The induced representation $\pi_1 \times \pi_2$ 
is then defined as the normalized parabolic induction
$$
\pi_1 \times \pi_2 = \Ind_{P_{n_1,n_2}}^{\G_{n_1+n_2}(F)} 
(\pi_1 \otimes \pi_2 \otimes \mathbbm{1}_{N_{n_1,n_2}}),
$$
where $\mathbbm{1}_{N_{n_1,n_2}}$ denotes the trivial character on the unipotent radical.

\subsection{Jacquet Functor} For a parabolic subgroup $P = MN$ of $\G_n(F)$ and a representation $\pi \in \Rep(\G_n(F))$, the normalized Jacquet module 
$r_P(\pi)$ is constructed as follows. Let $\delta_P$ denote the modulus 
character of $P$. Then $r_P(\pi)$ is the representation of $M$ given by 
the quotient
$$
r_P(\pi) = \delta_P^{-1/2} \cdot \pi / \langle \pi(n)v - v \mid v \in \pi, n \in N \rangle.
$$
This construction gives the 
normalized Jacquet functor, which is adjoint to the normalized parabolic induction functor.

\subsection{Cuspidal Support and Cuspidal Line} For every irreducible admissible representation $\pi$ of $\G_n(F)$ there exists a unique multiset $\{\rho_1,\rho_2,\ldots,\rho_r\}$ consisting of cuspidal representations such that $\pi$ occurs as a subquotient of the parabolically induce representation $\rho_1\times \rho_2\times \ldots\times \rho_r$. This distinguished multiset is called the \emph{cuspidal support} of $\pi$, denoted $\csupp(\pi)$. 

\begin{definition}
Given an irreducible representation $\pi$ of $\G_n(F)$, we denote \[ \csupp_{\mathbb{Z}}(\pi) = \{ \nu^{k}\rho | \rho\in \csupp(\pi), k\in \mathbb{Z} \} \] and say that elements of $\csupp_{\mathbb{Z}}(\pi)$ lie in the cuspidal line of $\pi$.    
\end{definition} 

\subsection{Bernstein Decomposition for $\GL_n(F)$} \label{s berstein decomp} We recall some basic material on the Bernstein decomposition and the theory of types. The Bernstein decomposition for $\G_n$ gives a decomposition
of the category of smooth representations of $\G_n$ into indecomposable
subcategories called \emph{Bernstein blocks}. 

A \emph{cuspidal pair} $(M, \sigma)$ consists of:
\begin{itemize}
\item A Levi subgroup $M$ of $\G_n$
\item An irreducible supercuspidal representation $\sigma$ of $M$.
\end{itemize} Two cuspidal pairs $(M_1, \sigma_1)$ and $(M_2, \sigma_2)$ are \emph{inertially equivalent} 
if there exists an unramified character $\chi$ of $M_2$ such that $\sigma_1 $ and $\chi \sigma_2$ are isomorphic after conjugation 
by an element of $\G_n$. The equivalence classes under this relation are called \emph{inertial classes}.

We let $\mathfrak{B}(\G_n)$ denote the set of all inertial equivalence classes in $\G_n$. Let $\Rep(\G_n)$ is the category of smooth $\G_n$-representations. The Bernstein decomposition gives a decomposition,
\begin{equation*}
\Rep(\G_{n}) \simeq \bigoplus_{\mathfrak s \in \mathfrak{B}(\G_{n})} \mathcal{R}_\mathfrak{s}(\G_{n}),
\end{equation*}
where $\mathcal{R}_\mathfrak{s}(\G_{n})$ is the full subcategory associated to the inertial class $\mathfrak{s}$. The subcategory $\mathcal{R}_\mathfrak{s}(\G_{n})$ is called the \emph{Bernstein block} associated to $\mathfrak{s}$.

For each $s \in \mathfrak{B}(\G_{n})$, we can associate a type $(K,\tau)$ where $K$ is an open compact subgroup and $(\tau,V)$ is a finite dimensional representation of $K$. We define the corresponding Hecke algebra as,
\[\mathcal{H}(K, \tau) := \{f : \G_n \to \operatorname{End}(V) : f(k g k') = \tau^{\vee}(k) f(g) \tau^{\vee}(k'), \, k, k' \in K\}\] For a smooth representation $\pi$ of $\G_n$, there is a natural action of $\mathcal{H}(K, \tau)$ on $\operatorname{Hom}_K(\tau,\pi)$. Then the map, \[ \begin{aligned}
     \mathcal{R}_\mathfrak{s}(\G_{n}) & \longrightarrow \mathcal{H}(K, \tau)-\rm{Mod} \\
    \pi & \longmapsto \Hom_K(\tau,\pi)
  \end{aligned} \] sending $\pi$ to $\Hom_K(\tau,\pi)$ gives an equivalence of categories between $\mathcal{R}_\mathfrak{s}(\G_{n})$ and the category of $\mathcal{H}(K, \tau)$ modules.

\begin{remark}
In this rest of the paper, given $\pi \in \mathcal{R}_\mathfrak{s}(G_{n})$ we shall frequently abuse notation and write $\pi$ instead of $\Hom_K(\tau,\pi)$ to be the associated $\mathcal{H}(K, \tau)$ module. It shall be clear from the context if we are treating $\pi$ as a representation of $\G_n$ or as a Hecke algebra module.
\end{remark} It is known that every affine Hecke algebra associated to a type for
$\GL_n(F)$ is isomorphic to the Iwahori–spherical Hecke algebra of some product of $\GL_{n_i}(F_i)$ such that $\sum n_i = n$ and each $F_i$ is a finite extension of the field $F$ \cite{bk99}. 

More specifically, for any inertial class $\mathfrak s$, the associated Hecke algebra $\mathcal H_{\mathfrak s}$ admits a tensor product decomposition, \[ \mathcal H_{\mathfrak s} = \mathcal H(n_1,q_1) \otimes \mathcal H(n_2,q_2) \otimes \cdots \mathcal H(n_k,q_k) \] where each $\mathcal H(n_i,q_i)$ is an affine Hecke algebra of type A.

We now recall the generators and relations for $\mathcal{H}(n,q)$, the affine Hecke algebra of type A. We may sometimes drop $q$ from the notation and simply write $\mathcal H_n = \mathcal{H}(n,q)$.

\subsection{Affine Hecke Algebra of Type A} 

The affine Hecke algebra of type $A$ is made up of two  subalgebras - one being a commutative ring of Laurent polynomials and the other being a finite hecke algebra of type A. We begin by introducing the finite Hecke algebra, followed by its affine extension.

\subsubsection{The Finite Hecke Algebra}
For a fixed parameter $q \in \mathbb{C}^\times$ (not a root of unity), the \emph{finite Hecke algebra} $\mathcal{H}_f(n, q)$ is the associative $\mathbb{C}$-algebra generated by $T_1, \dots, T_{n-1}$, subject to the following relations:

\begin{itemize}
    \item \textbf{Braid relations}:
    $
    T_i T_{i+1} T_i = T_{i+1} T_i T_{i+1}, \quad \text{for all } 1 \leq i \leq n-2,
    $
    
    \item \textbf{Quadratic relations}:
    $
    (T_i - q)(T_i + 1) = 0, \quad \text{for all } 1 \leq i \leq n-1,
    $
    
    \item \textbf{Commutation relations}:
    $
    T_i T_j = T_j T_i, \quad \text{for all } |j - i| > 1.
    $
\end{itemize}

\begin{remark}
    
Note that if $q=1$, then $\mathcal{H}_f(n, q)$ is nothing but the group algebra of the symmetric group. Therefore $ \mathcal{H}_f(n, q) $ can be thought of as a deformation of the group algebra of the symmetric group. \end{remark}

\subsubsection{Affine Hecke algebra}

We will define the complex affine Hecke algebra for $ \G_n $ and a parameter $q \in \mathbb{C}^\times$, to be the vector space 
$\mathcal{H}(n, q) $ defined as:
\[ \mathcal{H}(n, q) = \mathbb{C}[y_1^{\pm 1}, \dots, y_n^{\pm 1}] \otimes_{\mathbb{C}} \mathcal{H}_f(n, q), \] with the associative algebra structure given by both the subalgebras $ \mathbb{C}[y_1^{\pm 1}, \dots, y_n^{\pm 1}] $ (the commutative ring of Laurent polynomials) and $ \mathcal{H}_f(n, q) $, and the following commutation relations.

\begin{enumerate}
    \item $T_i y_i T_i = q y_{i+1}, \quad \forall \, 1 \leq i \leq n - 1,$
    \item $ T_i y_j = y_j T_i, \quad \forall \, j \neq i, i + 1.$
\end{enumerate}

\subsubsection{Center of the affine Hecke algebra} It is known that the center $\mathcal{Z}$ of the affine Hecke algebra is given by the $S_n$ invariant laurent polynomials, that is, \[ \mathcal{Z} = \mathbb{C}[y_1^{\pm 1}, \dots, y_n^{\pm 1}]^{S_n}. \] By the work of Lusztig \cite{lu89} it is known that $\mathcal Z$ has a basis given as, \[ z_M = \sum_{w\in S_n} y_1^{i_{w(1)}}y_2^{i_{w(2)}} \cdots y_n^{i_{w(n)}} \] where, $M$ runs over all tuples $(i_1,i_2,\ldots,i_n) \in \mathbb{Z}^n/S_n$.

\subsubsection{Notion of Parabolic Induction} There is a notion of parabolic induction for Hecke algebra modules as follows. Let $n_1,n_2$ be positive integers and set $n=n_1+n_2$. Let $\pi_1$ and $\pi_2$ be finitely generated modules of $\mathcal H_{n_1}$ and $\mathcal H_{n_2}$ respectively. The induced $\mathcal H_n$ module is defined as, \[ \pi = \mathcal{H}_n \otimes_{\mathcal H_{n_1} \otimes \mathcal H_{n_2}} (\pi_1 \otimes \pi_2), \] 
which is also finitely-generated as an $\mathcal H_n$ module.

\section{Representations of $\GL_n(F)$ and Bernstein-Zelevinsky Derivatives} \label{section - classification and bz derivatives}
We recall the classification of irreducible representations of $\GL_n(F)$ \cite{ze80} in terms of multisegments and fix our notation. We also discuss some classes of irreducible representations of $\GL_n(F)$ and recall the notion of Bernstein-Zelevinsky derivatives.
\subsection{Segments}
Consider $a, b \in \mathbb{C}$ such that $b-a\in \mathbb{Z}_{\geq 0}$. Given a cuspidal representation $\rho$ of $\G_n$, a segment $\Delta$ associated to the datum ($\rho$,$a$,$b$) is an ordered set of irreducible representations of $\G_n$ of the form, 
\[ \Delta = [a,b]_{\rho} = \{\nu^a\rho,\nu^{a+1}\rho,\ldots,\nu^b\rho\}. \] 
The endpoints of the segment $[a,b]_{\rho}$ are denoted as $a(\Delta)=a$ and let $b(\Delta)=b$. 

\begin{definition}
Two segments $\Delta_1 = [a, b]_{\rho}$ and $\Delta_2 = [c, d]_{\rho}$ are linked if
\begin{itemize}
    \item $a < c$ and $c - 1 \leq b < d$, or
    \item $c < a$ and $a - 1 \leq d < b$.
\end{itemize}
In the first case, we say that $\Delta_1$ precedes $\Delta_2$ and write $\Delta_1 \prec \Delta_2$. 
\end{definition}

\begin{definition}
Given the segment $\Delta=[a,b]_{\rho}$, the associated principal series, \[  \nu^a\rho\times \nu^{a+1}\rho\times \ldots \times \nu^b\rho\] has a unique irreducible submodule and unique irreducible quotient which we denote as $Z(\Delta)$ and $Q(\Delta)$ respectively.    
\end{definition}

\subsection{Multisegments} A multisegment $\mathfrak m$ is a finite collection of segments (with possible
multiplicities), $\mathfrak m = \Delta_1 + \Delta_2 + \cdots + \Delta_r$. Let $\operatorname{Mult}$ denote the set of multisegments.

\subsection{Langlands Classification} Given a multisegment $\mathfrak{m} = \{ \Delta_1,\Delta_2,\ldots,\Delta_r \}$ such that $\Delta_i\not \prec \Delta_j$ for all $i<j$, the standard module \[\lambda(\mathfrak m) = Q(\Delta_1)\times Q(\Delta_2) \times \ldots \times Q(\Delta_r)\] has a unique irreducible quotient. We set \[ Q(\mathfrak m) = \cosoc(\lambda(\mathfrak m)) \] where, $\cosoc(\lambda(\mathfrak m))$ denotes the cosocle (maximal semisimple quotient) of $\lambda(\mathfrak m)$. By the Langlands classification the map
\[ \mathfrak m \mapsto Q(\mathfrak m) \]
gives a bijection 
$ \operatorname{Mult} \longrightarrow \operatorname{Irr}$.

\subsection{Zelevinsky Classification}
Given a multisegment $\mathfrak{m} = \{ \Delta_1,\Delta_2,\ldots,\Delta_r \}$ such that $\Delta_i\not \prec \Delta_j$ for all $i<j$, the Zelevinsky standard module \[ \zeta(\mathfrak m) = Z(\Delta_1)\times Z(\Delta_2) \times \ldots \times Z(\Delta_r)\] has a unique irreducible submodule which we denote as $Z(\mathfrak{m})$. By the Zelevinsky classification any irreducible representation of $\GL_n(F)$ can be uniquely written in this manner. Hence the map
\[ \mathfrak m \mapsto Z(\mathfrak m) \]
gives a bijection 
$ \operatorname{Mult} \longrightarrow \operatorname{Irr}$.

\subsection{Speh Representations} \label{subsection on definition of speh and arthur type representations}

The Speh representations built out of a discrete series are defined as follows \cite{ta86}. Let $\rho \in \Irr$ be a cuspidal unitary representation. For any integer $a \geq 0$, we consider the discrete series representation $\delta_{\rho}(a)$ given as,
$$ \delta_{\rho}(a) = Q\left(\left[-\frac{a-1}{2}, \frac{a-1}{2}\right]_{\rho}\right). $$ For integers $a, b \geq 1$, the \emph{Speh representation} $u_{\rho}(a,b)$ can be constructed in two equivalent ways as follows:
\begin{itemize}
\item The unique irreducible quotient of the induced representation
$$
\nu^{\frac{b-1}{2}}\delta_{\rho}(a) \times \nu^{\frac{b-1}{2}-1}\delta_{\rho}(a) \times \cdots \times \nu^{-\frac{b-1}{2}}\delta_{\rho}(a)
$$

\item Equivalently, it can be obtained by taking the unique irreducible submodule of the induced representation,
$$
Z\left(\nu^{\frac{a-1}{2}}\Delta(\rho,b)\right) \times Z\left(\nu^{\frac{a-3}{2}}\Delta(\rho,b)\right) \times \cdots \times Z\left(\nu^{-\frac{a-1}{2}}\Delta(\rho,b)\right)
$$
where $\Delta(\rho,b) = \left[-\frac{b-1}{2}, \frac{b-1}{2}\right]_{\rho}$.
\end{itemize} The irreducible Speh representations $u_{\rho}(a,b)$ so constructed are known to be unitary.

\subsection{Arthur Parameters and Arthur type representations}

A representation is said to be of Arthur type if it arises from an Arthur parameter. We recall what this means below.

Let $W_F$ denote the Weil group of the field $F$. An Arthur parameter of $\GL_n(F)$ is an admissible homomorphism, \[ \psi: W_F \times \SL_2(\mathbb{C}) \times \SL_2(\mathbb{C}) \rightarrow \GL_n(\mathbb{C}) \] such that \begin{itemize}
    \item The restriction of $\psi$ to $W_F$ is an admissible homomorphism with a bounded image.
    \item The restriction of $\psi$ to two $\SL_2(\mathbb{C})$ factors is algebraic.
\end{itemize} The first $\SL_2(\mathbb{C})$ is called the Deligne $\SL_2$ and the second additional $\SL_2(\mathbb{C})$ is called the Arthur $\SL_2$. 

The construction of $L$-parameters from Arthur parameters proceeds in the following way. Consider the homomorphism
$
\alpha: W_F \times \SL_2(\mathbb{C}) \rightarrow W_F \times \SL_2(\mathbb{C}) \times \SL_2(\mathbb{C})
$
defined by
$$
\alpha(w,g) = \left(w, g, \begin{pmatrix}
    |w|^{1/2} & 0 \\
    0 & |w|^{-1/2}
\end{pmatrix}\right)
$$
for $w \in W_F$ and $g \in \SL_2(\mathbb{C})$, where $|w|$ denotes the canonical absolute value on the Weil group.

Given an Arthur parameter $\psi: W_F \times \SL_2(\mathbb{C}) \to \GL_n(\mathbb{C})$, we obtain the corresponding $L$-parameter via composition:
$$
\phi_\psi := \psi \circ \alpha: W_F \times \SL_2(\mathbb{C}) \to \GL_n(\mathbb{C})
$$
Then via the local Langlands correspondence:
\begin{enumerate}
\item Every Arthur parameter $\psi$ determines a unique irreducible unitary representation $\pi_\psi$ of $\GL_n(F)$
\item The representation $\pi_\psi$ depends only on the $\GL_n(\mathbb{C})$-conjugacy class of $\psi$.
\end{enumerate}

\begin{definition}
An irreducible unitary representation of $\GL_n(F)$ is said to be of \emph{Arthur type} if it arises from some Arthur parameter $\psi$ via this construction.
\end{definition}

We note that an irreducible representation $\pi$ of $\G_n$ is of Arthur type if there exist Speh representations $\pi_1,\pi_2,\cdots,\pi_k$ such that, \[ \pi = \pi_1 \times \pi_2 \times \cdots \times \pi_k. \] 

\subsection{Bernstein-Zelevinsky Derivatives and BZ Filtration} \label{subsection on derivatives}

Given $\pi\in \Rep(\G_n)$, we define the $i$-th Bernstein-Zelevinsky derivative $\pi^{(i)}$ as follows \cite{bz77}. Consider the subgroup $U_i < \G_{n-i}\times \G_i$ of upper triangular unipotent matrices in $\G_i$. Fix a non-degenerate character $\psi_i$ of $U_i$. We first introduce the following twisted jacquet module functor.

For any $\sigma \in \Rep(\G_{n-i} \times \G_i)$, we define:
$$
T_i(\sigma) := \sigma \big/ \left\langle \sigma(u)\cdot x - \psi_i(u)\cdot x \,\middle|\, x \in \sigma, u \in U_i \right\rangle
$$
This quotient construction yields $T_i(\sigma) \in \Rep(\G_{n-i})$.

Given $\pi\in \Rep(\G_n)$, the $i$-th Bernstein-Zelevinsky derivative $\pi^{(i)}\in \Rep(\G_{n-i})$ is defined as the composition,
\[ \pi^{(i)} = T_i \circ (r_{(n-i,i)}(\pi)). \]  Here, $r_{(n-i,i)}$ denotes the Jacquet module with respect to the standard parabolic of $\G_n$ that corresponds to the partition $(n-i,i)$ of $n$.
 
\subsubsection{Highest Derivative} Given, $\pi\in \Rep(\G_n)$, let $h$ be the largest integer such that $\pi^{(h)}\neq 0$. Then we say that $\pi^{(h)}$ is the highest derivative of $\pi$.

\subsubsection{Bernstein-Zelevinsky Filtration}

Let $\pi\in \Rep(\G_n)$. By the Bernstein-Zelevinsky filtration (see \cite{bz76}), we have a filtration of submodules $ 0=V_n\subsetneq V_{n-1}\subsetneq \cdots \subsetneq V_0 = \pi|_{\G_{n-1}}$ such that,
\[ V_i/V_{i+1} \cong \nu^{1/2}\pi^{(i+1)}\times \cind_{U_i}^{\G_i}\psi_i. \]
Here $U_i\subset \G_i$ is the subgroup of upper triangular unipotent matrices and $\psi_i$ is a fixed non-degenerate character of $U_i$.

\section{Full Faithfulness of the Parabolic Induction Functor on a completed category} \label{section - full faithfulness on a completed category}

We recall the notion of completed categories and some of their properties from \cite{fu18}. 
\subsection{Completed Categories} 
Let $\mathcal{C}$ be a $k$-linear abelian category equipped with a subfunctor $\phi \hookrightarrow \mathrm{Id}_{\mathcal{C}}$ that preserves epimorphisms. Then every object $V \in \mathcal{C}$ admits a natural filtration $ V \supset \phi(V) \supset \phi^2(V) \supset \cdots. $ 

\begin{definition}
We say that $\mathcal{C}$ is \textit{$\phi$-adically complete} if for every object $V \in \mathcal{C}$, the canonical map $ V \to \varprojlim V/\phi^n(V) $ is an isomorphism.    
\end{definition} 

We now describe the primary example of a completed category that will be central to our discussion.

\subsection{The Category $\widehat{\mathcal{C}}(\mathcal J)$} \label{subsection on the full faithfullness}

Given an irreducible representation $\pi$ in $\Rep(\mathrm{G}_m)$, let $\mathfrak{s}$ denote its inertial class and $\mathcal H$ denote the corresponding Hecke algebra. Let $\mathcal{J}$ be the maximal ideal in the center $\mathcal{Z}$ of $\mathcal H$ that annihilates $\pi$. Consider the completion $\widehat{\mathcal H}$ of $\mathcal H$ with respect to the maximal ideal $\mathcal J$ i.e. $\widehat{\mathcal H}=\varprojlim_i \mathcal H/\mathcal J^i\mathcal H$. 

Let $\omega \in \Rep(\G_k)$ be a Speh representation. Let $\mathcal C_\omega$ denote the full subcategory of finite-length representations $\sigma$ of $\mathrm{Rep}(\G_m)$ (for some $m$) and such that for any composition factor $\sigma^{\prime}$ of $\sigma$ and for any $\rho \in \csupp(\sigma^{\prime})$, either: \begin{enumerate}
    \item $\rho \in \csupp(\omega)$.
    \item $\rho \not \in (\csupp_\mathbb{Z}(\omega) - \csupp(\omega))$.
\end{enumerate}

Given an irreducible representation $\rho^{\prime}$ in $\mathcal C_{\omega}$, let $\mathfrak{s}$ denote its inertial class and $\mathcal H$ denote the corresponding Hecke algebra. Let $\mathcal{J}$ be the maximal ideal in the center $\mathcal{Z}$ of $\mathcal H$ corresponding to the irreducible representation $\rho^{\prime}$ lying in $C_\omega$ ; so the maximal ideal $\mathcal{J}$ annihilates $\rho^{\prime}$. We similarly let $\mathfrak s'$ be the inertial class of $\omega \times \rho'$ and let $\mathcal J'$ be the maximal ideal in the center $\mathcal Z'$ of the corresponding Hecke algebra $\mathcal H^{\prime}$ annihilating $\omega \times \rho'$. Let $\mathrm{Rep}_{\mathfrak s}(\mathrm G_m)$ and $\mathrm{Rep}_{\mathfrak s'}(\mathrm G_m)$ be the Bernstein blocks of $\mathrm{Rep}(\mathrm G_m)$ associated to $\mathfrak s$ and $\mathfrak s'$ respectively.

Let $\mathcal C(\mathcal{J})$ (resp. $\mathcal C(\mathcal J')$) be the category of representations in $\mathrm{Rep}(\mathrm G_m)$ (resp. $\mathrm{Rep}(\mathrm G_{m+k})$ that are annihilated by a power of the maximal ideal $\mathcal J$ (resp. $\mathcal J'$). The objects in this category are of finite length. Let us also consider the completed category $\widehat{C}(\mathcal{J})$ whose objects are of the form,
\begin{align} \label{eqn tau complete}
\widehat{\tau} = \varprojlim(\tau/\mathcal{J}^i) 
\end{align}
for some object $\tau\in \mathrm{Rep}_\mathfrak{s}(\mathrm G_{n})$ such that $\tau/\mathcal J^i$ is of finite length. By the equivalence of cateogries in \cite{bk99}, we shall usually regard objects in $\widehat{\mathcal C}(\mathcal J)$ as $\mathcal H$-modules.

We define an endofunctor $\phi$ on $\widehat{\mathcal{C}}(\mathcal J)$ by $\phi(V) := \mathcal{J}V$. Then $\phi$ is naturally a subfunctor of $\mathrm{Id}_{\mathcal{C}}$ preserving epimorphisms. As objects in $\widehat{\mathcal C}(\mathcal J)$ are finitely-generated $\widehat{\mathcal H}$-modules, $\widehat{\mathcal{C}}(\mathcal J)$ is a $\phi$-adically complete category. We similarly define the completed category $\widehat{\mathcal C}(\mathcal J')$. 

We now compare $\widehat{\mathcal C}(\mathcal J)$ with the category $\mathrm{Mod}_{fg}(\widehat{\mathcal H})$ of finitely-generated $\widehat{\mathcal H}$-modules.

\begin{lemma} \label{lem H hat morphisms}
Let $\pi$ be a finitely generated $\widehat{\mathcal H}$-module. Then 
\[   \mathrm{Hom}_{\widehat{\mathcal C}(\mathcal J)}(\widehat{\mathcal H}, \pi) \cong \mathrm{Hom}_{\widehat{\mathcal H}}(\widehat{\mathcal H}, \pi) .
\]
\end{lemma}

\begin{proof}
For $f \in \mathrm{Hom}_{\widehat{\mathcal H}}(\widehat{\mathcal H}, \pi)$, one defines a $\mathcal H$-module morphism by restricting to $\mathcal H$-structure. Conversely, for $f' \in \mathrm{Hom}_{\widehat{\mathcal C}(\mathcal J)}(\widehat{\mathcal H}, \pi)$, one obtains $\mathcal H$-morphisms $f_k$ via composition:
\[    \widehat{\mathcal H} \stackrel{f'}{\rightarrow} \pi \rightarrow \pi/\mathcal J^k.\pi ,
\]
where the second morphism is the natural projection. Note that $\mathcal J^k\widehat{\mathcal H}$ lies in the kernel of the composition, and so it also lifts to a $\widehat{\mathcal H}$-morphism. By the universal property of inverse limit, we obtain a morphism in $\mathrm{Hom}_{\widehat{\mathcal C}(\mathcal J)}(\widehat{\mathcal H}, \pi)$. One can check that the above two construction are inverse to each other.
\end{proof}

\begin{lemma} \label{lem equiv cate}
There is an equivalence of categories between $\widehat{\mathcal C}(\mathcal J)$ and the category $\mathrm{Mod}_{fg}(\widehat{\mathcal H})$ of finitely-generated $\widehat{\mathcal H}$-modules.
\end{lemma}

\begin{proof}
Since any objects in $\widehat{\mathcal C}(\mathcal J)$ are finitely-generated $\widehat{\mathcal H}$-modules, we can regard objects in $\widehat{\mathcal C}(\mathcal J)$ as objects in $\mathrm{Mod}_{fg}(\widehat{\mathcal H})$. It remains to check that 
\[   \mathrm{Hom}_{\widehat{\mathcal C}(\mathcal J)}(\pi_1, \pi_2) \cong \mathrm{Hom}_{\widehat{\mathcal H}}(\pi_1, \pi_2) .
\]
By the Noetherian property of $\widehat{\mathcal H}$-modules, we have an exact sequence from the first two terms of a free $\widehat{\mathcal H}$-resolution:
\[   \widehat{\mathcal H}^{\oplus s} \rightarrow  \widehat{\mathcal H}^{\oplus r} \rightarrow \pi_1  .
\]
Note that we have a natural inclusion form RHS to LHS by only considering the $\mathcal H$-actions from $\widehat{\mathcal H}$-actions.

Now we have the following commutative diagram, where the two horizontal rows are exact:
\[  \xymatrix{ 0 \ar[r] &  \mathrm{Hom}_{\widehat{\mathcal H}}(\pi_1, \pi_2) \ar[r] \ar[d] & \mathrm{Hom}_{\widehat{\mathcal H}}( \widehat{\mathcal H}^{\oplus r}, \pi_2) \ar[r]^{f} \ar[d] & \mathrm{Hom}_{\widehat{\mathcal H}}( \widehat{\mathcal H}^{\oplus s}, \pi_2) \ar[d]   \\
 0 \ar[r] &  \mathrm{Hom}_{\widehat{\mathcal C}(\mathcal J)}(\pi_1, \pi_2) \ar[r] & \mathrm{Hom}_{\widehat{\mathcal C}(\mathcal J)}( \widehat{\mathcal H}^{\oplus r}, \pi_2) \ar[r]^{f'} & \mathrm{Hom}_{\widehat{\mathcal C}(\mathcal J)}( \widehat{\mathcal H}^{\oplus s}, \pi_2)  
}
\]
As discussed in Lemma \ref{lem H hat morphisms}, the vertical maps are those natural inclusions by restricting to $\mathcal H$-actions. Now it follows from Lemma \ref{lem H hat morphisms} that the kernel of $f$ is isomorphic to the kernel of $f'$. This shows that the leftmost vertical map is an isomorphism.
\end{proof}

In terms of studying Ext-groups in $\mathcal R_{\mathfrak s}(\mathrm{G}_m)$, the objects in $\widehat{\mathcal C}(\mathcal J)$ are not the ones of most interest, as most representations in $\widehat{\mathcal C}(\mathcal J)$ are not finitely-generated as $\mathrm{G}_m$-representations. Instead, finitely generated representations in $\mathrm{Rep}(\mathrm{G}_m)$ are the ones that naturally arise in branching laws. The connection from $\mathcal R_{\mathfrak s}(\mathrm{G}_m)$ to $\widehat{\mathcal C}(\mathcal J)$ is as follows:

\begin{proposition} \label{prop ext identify}
Let $\pi_1$ and $\pi_2$ be finitely-generated representations in $\mathcal R_{\mathfrak s}(\mathrm{G}_m)$. We view $\pi_1$ and $\pi_2$ as $\mathcal H$-modules as in Section \ref{s berstein decomp}, and then there is a natural $\mathcal Z$-action on $\mathrm{Ext}^i_{\mathcal H}(\pi_1, \pi_2)$. Recall that $\widehat{\pi_1}$ and $\widehat{\pi_2}$ are defined as in (\ref{eqn tau complete}). Then, for any $i \geq 0$, 
\[  \mathrm{Ext}^i_{\mathcal H}(\pi_1, \widehat{\pi_2}) \cong \widehat{\mathcal Z}\otimes_{\mathcal Z} \mathrm{Ext}^i_{\mathcal H}(\pi_1, \pi_2) \cong \mathrm{Ext}^i_{\widehat{\mathcal C}(\mathcal J)}(\widehat{\pi}_1, \widehat{\pi}_2) .
\]
\end{proposition}

\begin{proof}
For the former isomorphism, one sees a result of Nori and Prasad \cite[Proposition 5.2]{np20}. Indeed, the key for the isomorphism is that $\pi_1$ admits a free $\mathcal H$-resolution, and there are natural identifications between the following spaces:
\[   \widehat{\mathcal Z} \otimes_{\mathcal Z}  \mathrm{Hom}_{\mathcal H}(\mathcal H, \pi_2) \cong \widehat{\mathcal Z}\otimes_{\mathcal Z}\pi_2=\widehat{\pi_2} \cong  \mathrm{Hom}_{\mathcal H}(\mathcal H, \widehat{\pi_2}) .
\]

We now explain the latter isomorphism. It follows from \cite[Lemma A.5]{cs19} that there is a natural isomorphism:
\[  \widehat{\mathcal Z}\otimes_{\mathcal Z} \mathrm{Ext}^i_{\mathcal H}(\pi_1, \pi_2) \cong \mathrm{Ext}^i_{\widehat{\mathcal H}}(\widehat{\pi}_1, \widehat{\pi}_2) .
\]
Now the proposition follows from Lemma \ref{lem equiv cate}.
\end{proof}

\begin{corollary} \label{cor non-vanishing}
We keep the setting in Proposition \ref{prop ext identify}. Suppose furthermore that $\pi_1$ or $\pi_2$ is of finite length and is in $\widehat{\mathcal C}(\mathcal J)$ i.e. either $\pi_1 \cong \pi_1/\mathcal J^k\pi_1$ and $\pi_2\cong \pi_2/\mathcal J^k\pi_2$ for some $k \geq 0$. If $\mathrm{dim}~\mathrm{Ext}^i_{\mathcal H}(\pi_1, \pi_2) <\infty$, then  
\[  \mathrm{Ext}^i_{\mathcal H}(\pi_1, \pi_2) \cong \mathrm{Ext}^i_{\widehat{\mathcal C}(\mathcal J)}(\widehat{\pi}_1, \widehat{\pi}_2) .
\]
\end{corollary}

\begin{proof}
Since $\mathrm{dim}~\mathrm{Ext}^i_{\mathcal H}(\pi_1, \pi_2)<\infty$ and one of $\pi_1$ and $\pi_2$ is in $\widehat{\mathcal C}(\mathcal J)$, it implies that the Ext space is annihilated by some power of $\mathcal J$. Thus $\widehat{\mathcal Z}\otimes_{\mathcal Z}\mathrm{Ext}^i_{\mathcal H}(\pi_1, \pi_2)\cong \mathrm{Ext}^i_{\mathcal H}(\pi_1, \pi_2)$. Now the corollary follows from Proposition \ref{prop ext identify}.
\end{proof}

\subsection{Statement of the Result}
We first recall the following result from \cite{ch24product} that is about representations of finite length. 

\begin{theorem} \label{theorem on ff functor for finite length}
Let $\tau_1$ and $\tau_2$ belong to the category $\mathcal C(\mathcal J)$. Then,
 \[ \Hom_{\mathcal C(\mathcal J)}(\tau_1, \tau_2)\cong \Hom_{\mathcal C(\mathcal J')}(\omega \times \tau_1, \omega \times \tau_2). \] Here the above isomorphism is given by sending $\phi$ to $ \mathrm{Id} \times \phi$. 
 \end{theorem}

We aim to prove the following result, which is the main theorem of this section.

\begin{theorem} \label{theorem on fully faithful functor for infinte length}
Let $\widehat{\tau_1}$ and $\widehat{\tau_2}$ belong to the category $\widehat{\mathcal C}(\mathcal{J})$. Then,
 \[ \mathrm{Hom}_{\widehat{\mathcal C}(\mathcal J)}(\widehat{\tau_1}, \widehat{\tau_2})\cong \mathrm{Hom}_{\widehat{\mathcal C}(\mathcal J')}(\omega \times \widehat{\tau_1}, \omega \times \widehat{\tau_2}). \] where the above isomorphism is given by sending $\phi$ to $\mathrm{Id} \times\phi$.    
\end{theorem} In order to prove that the map above is an isomorphism we deal with the injectivity and surjectivity separately. 

\subsection{Proof of Injectivity}

We first recall the following lemma from category theory.

\begin{lemma} \cite[Chapter 2, Proposition 7.2]{mi65}
Let $\mathcal C_1$ and $\mathcal C_2$ be two abelian categories and let $F: \mathcal C_1 \rightarrow \mathcal C_2$ be a functor satisfying the following two conditions: \begin{enumerate}
    \item $F$ is a an exact functor.
    \item For any object $A$ in $(\mathcal C_1)$, if $F(A)=0$ then $A=0$.
\end{enumerate} Then, the functor $F$ is faithful.
\end{lemma} Since parabolic induction is an exact functor which takes non-zero objects to non-zero objects, from the above lemma we conclude that, \[ \mathrm{Hom}_{\widehat{\mathcal C}(\mathcal J)}(\widehat{\tau_1}, \widehat{\tau_2})\hookrightarrow \mathrm{Hom}_{\widehat{\mathcal C}(\mathcal J')}(\omega \times \widehat{\tau_1}, \omega \times \widehat{\tau_2}). \]

\subsection{Proof of Surjectivity}

We want to show that, \[ \mathrm{Hom}_{\widehat{\mathcal C}(\mathcal J)}(\widehat{\tau_1}, \widehat{\tau_2})\twoheadrightarrow \mathrm{Hom}_{\widehat{\mathcal C}(\mathcal J')}(\omega \times \widehat{\tau_1}, \omega \times \widehat{\tau_2}). \]

Let $\phi \in \mathrm{Hom}(\omega \times \widehat{\tau_1}, \omega \times \widehat{\tau_2})$. As parabolic induction commutes with the inverse limit (see the appendix for some details), we have:
\[   \mathrm{Hom}_{\widehat{\mathcal C}(\mathcal J')}(\omega \times \widehat{\tau}_1, \omega \times \widehat{\tau}_2) \cong \mathrm{Hom}_{\widehat{\mathcal C}(\mathcal J')}(\widehat{\omega \times \tau_1}, \widehat{\omega \times \tau_2}) .
\]

Thus we may view the map $\phi \in \mathrm{Hom}_{\widehat{\mathcal C}(\mathcal J')}(\widehat{\omega \times \tau_1}, \widehat{\omega \times \tau_2})$. Note that   
\[  (\omega \times \tau_1)/\mathcal J' \cong \omega \times (\tau_1/\widetilde{\mathcal J})
\]
for some ideal $\widetilde{\mathcal J}$ of $\mathcal Z$. Note that $\mathcal J'$ annihilates all irreducible objects with the corresponding central characters. We also have $\widetilde{\mathcal J}\subset \mathcal J$ by using the condition that $\mathcal J'$ annhilates those irreducible objects, and also $\mathcal J^k\subset \widetilde{\mathcal J}$ for some $k$ by using $\cap_r \mathcal J^r=0$. Then, we accordingly have that the $\mathcal J$-adic completion of $\tau_i$ is isomorphic to $\widetilde{\mathcal J}$-adic completion of $\tau_i$.

The above gives us the natural maps, \[ \phi_i: \omega \times (\tau_1/\mathcal{J}^i) \rightarrow \omega \times (\tau_2/\mathcal{J}^i) . \]
Now invoking Theorem \ref{theorem on ff functor for finite length} (\cite[Theorem 1.2]{ch24product}), we conclude that each $\phi_i$ can be written as, \[ \phi_i = \mathrm{Id} \times \psi_i. \] Since the maps $\{ \phi_i \}$ form a compatible system, we can check that the maps $\psi_i$ also form a compatible system of maps.

For if, \[ \phi_{k,j} \circ \phi_{i,k} = \phi_{i,j} \] \[ \implies \mathrm{Id} \times \psi_{k,j} \circ \psi_{i,k} = \mathrm{Id} \times \psi_{i,j}. \] Then invoking Theorem \ref{theorem on ff functor for finite length} we have that, \[ \psi_{k,j} \circ \psi_{i,k} = \psi_{i,j}. \] Hence the maps, $\psi_i: \tau_1/\mathcal{J}^i \rightarrow \tau_2/\mathcal{J}^i$, give rise to a map \[ \psi: \widehat{\tau_1} \rightarrow \widehat{\tau_2}.\] This completes the proof of surjectivity.

\section{On Extensions and Parabolic Inductions} \label{section - on extensions and parabolic induction}

In this section, we shall extend the result of product functor to a completed category.  

\subsection{Embedding Theorem}
Let us fix notations as in Section \ref{subsection on the full faithfullness}. In this section we prove the following. 

\begin{theorem} \label{theorem on extensions and parabolic induction}
Let $\tau_1$ and $\tau_2$ be objects in the category $\widehat{C}(\mathcal{J})$. Then, there is an injection
 \[ \mathrm{Ext}_{\widehat{C}(\mathcal{J})}^i(\tau_1, \tau_2)\hookrightarrow \mathrm{Ext}_{\widehat{C}(\mathcal{J}')}^i(\omega \times \tau_1, \omega \times \tau_2). \] for all integers $i\geq 0$.    
\end{theorem}

\subsection{Ideas for $i=1$ in Theorem \ref{theorem on extensions and parabolic induction}}
We first illustrate some ideas for $i=1$ and $\tau_1$ and $\tau_2$ are of finite length. Indeed, if we start with a short exact sequence in $\widehat{\mathcal C}(\mathcal J)$:
\begin{align} \label{eqn split short exact}
0 \rightarrow  \tau_1  \rightarrow  \tau \rightarrow \tau_2 \rightarrow 0 ,
\end{align}
producing with $\omega$ gives the short exact sequence:
\[  0 \rightarrow \omega \times \tau_1 \rightarrow \omega \times \tau \rightarrow \omega \times \tau_2 \rightarrow 0 .
\]
If the above sequence splits, then we obtain a commutative diagram:
\[  \xymatrix{  0  \ar[r] &  \omega \times  \tau_1 \ar[r] \ar@{=}[d] &  \omega\times \tau \ar[r] \ar[d]^{\cong} &  \omega \times \tau_2 \ar[r] \ar@{=}[d] & 0   \\ 
0 \ar[r] & \omega \times \tau_1 \ar[r] & \omega \times \tau_1\oplus \omega \times \tau_2 \ar[r] & \omega \times \tau_2 \ar[r] & 0  
} .
\]
By applying Theorem \ref{theorem on fully faithful functor for infinte length} and the adjointness of $\omega \times$, one obtains that (\ref{eqn split short exact}) splits.

We now look at this from the viewpoint of projective resolutions, and illustrate the general idea of the following proof. We consider a projective resolution in $\widehat{\mathcal C}(\mathcal J)$:
\[     \ldots \stackrel{d^P_2}{\rightarrow} P_1 \stackrel{d^P_1}{\rightarrow} P_0 \stackrel{d^P_0}{\rightarrow} \tau_1 \rightarrow 0 ,
\]
and similarly we also have a projective resolution in the corresponding completed category,
\[                \ldots      \stackrel{d^Q_2}{\rightarrow}   Q_1     \stackrel{d^Q_1}{\rightarrow}  Q_0       \stackrel{d^Q_0}{\rightarrow}  \omega \rightarrow 0 .
\]
Let $Q_0'$ be the kernel of the map $d_0^Q$. Let $\widetilde{d}^P_1$ be the induced map from $d^P_1$. We consider the following commutative diagram:
\[ \xymatrix{    0 \ar[d]  &  0 \ar[d]  &   &    \\
Q_1\times P_0 +Q_0'\times P_1 \ar@{->>}[r] \ar@{^{(}->}[d] & Q_0' \times P_0 \ar@{^{(}->}[d]  &   \\
                Q_1\times P_0+Q_0\times P_1 \ar[r]^s \ar@{->>}[d]^h & Q_0 \times P_0 \ar[r] \ar@{->>}[d]^{h'} \ar@{-->}[ldd]^{f'} & \omega \times \tau_1 \ar[r] \ar[d]  & 0 \\ 
                \omega \times P_1 \ar[r]^{\widetilde{d}_1^P} \ar@{.>}[d]^f    &    \omega \times P_0 \ar[r] \ar@{-->}[ld]^{f''}  &  \omega \times \tau_1 \ar[r] & 0 \\
                \omega \times \tau_2
                }
\]
The above maps are the natural ones, and for example, the top horizontal map is the projection onto the second factor.

Given $\eta \in \mathrm{Ext}^1_{\widehat{\mathcal C}(\mathcal J)}(\tau_1, \tau_2)$, it can be represented by a map $\underline{f}$ from $P_1$ to $\tau_2$. This induces the map $f=\mathrm{Id} \times \underline{f}:\omega \times P_1\rightarrow \omega \times \tau_1$. Then the composition $f\circ h$ represents an element in $\mathrm{Ext}^1_{\widehat{\mathcal C}(\mathcal J')}(\omega \times \tau_1,\omega \times \tau_2)$.

Suppose there exists $f': Q_0 \times P_0 \rightarrow \omega \times \tau_1$ such that the bottom-left triangle commutes i.e.
\[    f \circ h=f' \circ s . 
\]
In other words, $f\circ h$ represents a zero element in $\mathrm{Ext}^1_{\widehat{\mathcal C}(\mathcal J')}(\omega \times \tau_1, \omega \times \tau_2)$. Note that the solid leftmost arrows form an exact sequence, and so the surjectivity implies that the map $Q_0' \times P_0 \subset \mathrm{ker}~f'$. This then implies that $f'$ factors through $h'$ i.e. there exists $f'': \omega \times P_0 \rightarrow \omega \times \tau_2$ such that $f'=f'' \circ h'$. 

The upshot is that now applying the functorial properties of parabolic inductions, one then has that 
\[   \underline{f}= \underline{f}'' \circ d_1^P 
\]
for some $\underline{f}'': P_0\rightarrow \tau_2$, and hence $\underline{f}$ represents a zero element in $\mathrm{Ext}^1_{\widehat{\mathcal C}(\mathcal J)}(\tau_1, \tau_2)$.

\subsection{Proof of Theorem \ref{theorem on extensions and parabolic induction}} \label{ss thm extensions parabolic ind}

Let us consider a projective resolution of $\tau_1$ in $\widehat{\mathcal C}(\mathcal J)$ as follows,
$$\tau_1 \xleftarrow{ d } P_0 \xleftarrow{ d_0 } P_1 \xleftarrow{ d_1 } P_2 \xleftarrow{ d_2 } P_3 \xleftarrow{ d_3 } \cdots $$ where the projective modules $P_i$ are finitely generated. We form the cochain complex:
$$ 0 \xrightarrow{} \mathrm{Hom}_{\widehat{\mathcal C}(\mathcal J)}(P_0, \tau_2) \xrightarrow{d_0^*} \mathrm{Hom}_{\widehat{\mathcal C}(\mathcal J)}(P_1, \tau_2) \xrightarrow{d_1^*} \mathrm{Hom}_{\widehat{\mathcal C}(\mathcal J)}(P_2, \tau_2) \xrightarrow{d_2^*} \mathrm{Hom}_{\widehat{\mathcal C}(\mathcal J)}(P_3, \tau_2) \xrightarrow{d_3^*} \cdots \, $$
Now we are given that,
$$\mathrm{Ext}^n_{\widehat{\mathcal C}(\mathcal J)}(\tau_1, \tau_2) = \ker(d_n^*) / \mathrm{Im}(d_{n-1}^*) \neq 0\,.$$
Fix a non-zero element $f \in \mathrm{Ext}^n_{\widehat{\mathcal C}(\mathcal J)}(\tau_1, \tau_2)$. Therefore, $f \in \mathrm{Hom}_{\widehat{\mathcal C}(\mathcal J)}(P_n, \tau_2)$ such that $f \circ d_n = 0$ and $f \neq g \circ d_{n-1}$ for any $g \in \mathrm{Hom}_{\widehat{\mathcal C}(\mathcal J)}(P_{n-1}, \tau_2)$. Now consider a projective resolution of $\omega$ given as follows, $$ \omega \xleftarrow{e} Q_0 \xleftarrow{ e_0 } Q_1 \xleftarrow{ e_1 } Q_2 \xleftarrow{ e_2 } Q_3 \xleftarrow{ e_3 } \cdots \,. $$ We obtain a projective resolution of $\omega \times \tau_1$ gives as,
$$\omega \times \tau_1 \xleftarrow{ h } Q_0 \times P_0 \xleftarrow{ h_0 } Q_1 \times P_0 + Q_0 \times P_1 \xleftarrow{ h_1 } Q_2 \times P_0 + Q_1 \times P_1 + Q_0 \times P_2 \xleftarrow{ h_2 } \cdots \,,$$
where the differentials $h_i$ are defined as:
\begin{align*}
& h = e \times d, \\
& h_0 = e_0 \times \mathrm{Id} + \mathrm{Id} \times d_0, \\
& h_1 = e_1 \times \mathrm{Id} + (e_0 \times \mathrm{Id} - \mathrm{Id} \times d_0) + \mathrm{Id} \times d_1,
\end{align*} and so on similarly for higher degrees. Similarly, for this projective resolution of $ \omega \times \tau_1$, we have the following cochain complex:
\begin{align*}
\mathrm{Hom}(Q_0\times P_1+ Q_1\times P_0, \omega \times \tau_2) 
 \xleftarrow{h_0^*} & \mathrm{Hom}(Q_0\times P_0, \omega \times \tau_2) 
 \xleftarrow{h^*} 0 \\
\ldots \xleftarrow{h^*_2}  \mathrm{Hom}(Q_0 \times P_2 +Q_1\times P_1 & + Q_2\times P_0, \omega\times \tau_2 ) 
 \xleftarrow{h^*_1} \\
\mathrm{Hom}(Q_0 \times P_{n} + \ldots +  Q_{n} \times P_0,  \omega \times \tau_2) &
 \xleftarrow{h^*_{n-1}}  \mathrm{Hom}(Q_0\times P_{n-1} + \ldots + Q_{n-1}\times P_0, \omega \times \tau_2) \\
\mathrm{Hom}(Q_0 \times P_{n+1} + \ldots +& Q_{n+1} \times P_0,  \omega \times \tau_2) 
 \xleftarrow{h^*_n}
\end{align*} We want to construct a non-zero element $\tilde{f} \in \ker(h^*_n)$ such that $\tilde{f} \notin \mathrm{Im}(h^*_{n-1})$. This means we need to construct a non-zero element
$\tilde{f} \in \mathrm{Hom}\left( Q_0 \times P_{n} + \cdots + Q_{n} \times P_0, \omega \times \tau_2 \right)$ satisfying: $$\tilde{f} \circ h_n = 0,$$ and such that for any $\tilde{g} \in \mathrm{Hom}(Q_0\times P_{n-1}  +  Q_1 \times P_{n-2} + \ldots + Q_{n-1} \times P_0, \omega \times \tau_2)$ 
$$\tilde{f} \neq \tilde{g} \circ h_{n-1}$$

\subsubsection{Defining the map} We define $\tilde{f} :  Q_0 \times P_{n} + \cdots + Q_{n}  \times P_0 \to  \omega \times \tau_2$ as follows:
$$
\tilde{f}(\mu) = 
\begin{cases}
(e \otimes f) \circ \mu & \text{if } \mu \in  Q_0 \times P_n, \\
0 & \text{if } \mu \in  Q_1 \times P_{n-1} + \cdots +  Q_n \times P_0.
\end{cases}
$$

\subsubsection{Verifying that $\tilde{f}$ lies in the kernel of $h_n^*$} We check that $\tilde{f} \circ h_n$ vanishes on all components:

\begin{enumerate}
    \item[(i)] On $ Q_0 \times P_{n+1}$: 
    For any $\mu \in Q_0 \times P_{n+1}$ and $g \in \mathrm{GL}_n$ with $\mu(g) = q_0 \otimes p_{n+1}$ where $p_{n+1} \in P_{n+1}$, $q_0 \in Q_0$, we have:
    \begin{align*}
        \tilde{f} \circ h_n \circ \mu(g) 
        &=  e(q_0) \otimes f \circ d_n(p_{n+1}) \\
        &= 0 \quad \text{(since $f \circ d_n = 0$)}
    \end{align*} Hence $\tilde{f} \circ h_n \big|_{P_{n+1} \times Q_0} = 0.$
    
    \item[(ii)] On $ Q_1 \times P_n$: Similarly we have that
    \begin{align*}
        \tilde{f} \circ h_n \big|_{ Q_1 \times P_n} = 0
    \end{align*}
    because $\tilde{f}|_{ Q_1 \times P_{n-1}} = 0$ and $e \circ e_0 = 0$.
    
    \item[(iii)] On $Q_2 \times P_{n-1} + \cdots + Q_{n+1} \times  P_0$: We have that
    \begin{align*}
        \tilde{f} \circ h_n \big|_{ Q_2 \times P_{n-1} + \cdots + Q_{n+1} \times P_0} = 0
    \end{align*}
    since by definition $\tilde{f}|_{ Q_1 \times P_{n-1} + \cdots + Q_n \times P_0} = 0$.
\end{enumerate}

Therefore, we conclude that:
\begin{equation*}
    \tilde{f} \circ h_n = 0.
\end{equation*}

\subsubsection{Verifying that $\tilde{f}$ does not lie in the image of $h_{n-1}^*$} Let us assume for the sake of contradiction, that there exists an element  
$\tilde{g} \in \mathrm{Hom}\left( Q_0 \times P_{n-1} + \cdots + Q_{n-1} \times P_0, \omega \times \tau_2\right)$
such that
\begin{equation} \label{eqn one in embedding theorem}
\tilde{f} = \tilde{g} \circ h_{n-1}.    
\end{equation}
Since $\tilde{f}|_{ Q_1 \times P_{n-1} + \cdots + Q_n \times P_0} = 0$, by equation \ref{eqn one in embedding theorem} we have that $\tilde{g} \circ h_{n-1}$ descends to an element in $\mathrm{Hom}\left(Q_0 \times P_{n}, \omega \times \tau_2\right)$.
Recall that we have the short exact sequence, 
\[ 0 \rightarrow \ker(e) \rightarrow Q_0 \xrightarrow{e} \omega \rightarrow 0 \] 
 and so we get, 
\begin{equation} \label{eqn two in embedding theorem}
0 \rightarrow  \ker(e)  \times P_n\rightarrow Q_0  \times P_n\rightarrow \omega  \times P_n\rightarrow 0    
\end{equation} Since $\tilde{f}|_{\ker(e) \times P_n} = 0$, $\tilde{f}$ descends to an element in $\mathrm{Hom}(\omega \times P_{n},\omega \times \tau_2)$ that by abuse of notation, we still denote as $\tilde{f}$. Also, because $\tilde{f}|_{\ker(e) \times P_n} = 0$, by equation \ref{eqn one in embedding theorem} and \ref{eqn two in embedding theorem} we conclude that $\tilde{g}\circ h_{n-1}$ also descends to an element in $\mathrm{Hom}(\omega \times P_{n},\omega \times \tau_2)$ via the map $e \times \mathrm{Id}$. 

Note that $h_{n-1}|_{Q_0 \times P_{n}} = \mathrm{Id} \times d_{n-1}$. Combining with equation \ref{eqn one in embedding theorem} and since $P_n$ is in $\widehat{C}(\mathcal J)$, we can apply Theorem \ref{theorem on fully faithful functor for infinte length} to conclude that $f = \tilde{g}^{\prime}\circ d_{n-1}$ for some $\tilde{g}^{\prime}$ which is the desired contradiction.

We thus obtain a nonzero mapping, \[ \phi \colon \operatorname{Ext}_{\widehat{\mathcal C}(\mathcal J)}^n(\tau_1, \tau_2) \hookrightarrow \operatorname{Ext}_{\widehat{\mathcal C}(\mathcal J')}^n(\omega \times \tau_1, \omega \times \tau_2 )
\] defined by $\phi(f) = \tilde{f}$. The construction immediately shows that $f \neq 0$ if and only if $\tilde{f} \neq 0$, which establishes the injectivity of $\phi$. 
\qed

\subsection{An Ext non-vanishing result} We obtain the following Ext non-vanishing result as a corollary of Theorem \ref{theorem on extensions and parabolic induction}.

\begin{corollary} \label{corollary on nonzeroness 1}
Let $\tau_1$ and $\tau_2$ belong to the category $\widehat{C}(\mathcal{J})$. Then,
 \[ \mathrm{Ext}_{\widehat{\mathcal C}(\mathcal J)}^i(\tau_1, \tau_2)\neq 0 \implies  \mathrm{Ext}_{\widehat{\mathcal C}(\mathcal J')}^i(\omega \times \tau_1, \omega \times \tau_2)\neq 0 \] for all integers $i\geq 0$.        
\end{corollary}

\subsection{Proof of Theorem \ref{theorem on injective mapping on ext modules}}





We now prove a slightly stronger version of Theorem \ref{theorem on injective mapping on ext modules}:

\begin{theorem} 
We use the notations as in Section \ref{subsection on the full faithfullness}. Let $\tau_1$ and $\tau_2$ belong to the category $\widehat{\mathcal C}(\mathcal J)$. Suppose either $\tau_1$ or $\tau_2$ is of finite-length, or more generally suppose that $\mathrm{dim}~\mathrm{Ext}^i_{\mathrm{G}_m}(\tau_1, \tau_2)<\infty$. Then, for any $i$, there is an injection
 \[ \phi: \mathrm{Ext}_{\G_m}^i(\tau_1, \tau_2)\hookrightarrow \mathrm{Ext}_{\G_{m+k}}^i(\omega \times \tau_1, \omega \times \tau_2) . \]
\end{theorem}

\begin{proof}
Again, we can view $\tau_1$ and $\tau_2$ as $\mathcal H$-modules. As $\mathcal H$ is noetherian, $\tau_1$ admits a finitely-generated free $\mathcal H$-resolution. Now, as the space $\mathrm{Hom}_{\mathcal H}(\mathcal H, \tau_2)\cong \tau_2$, and so $\mathrm{Ext}^i_{\mathcal H}(\tau_1, \tau_2)$ is also finitely-generated.  Since $\tau_1$ or $\tau_2$ is of finite length, we have that either $\tau_1$ or $\tau_2$ is annihilated by some powers of $\mathcal J$. We then also have that the Ext-groups $\mathrm{Ext}^i_{\mathrm{G}_m}(\tau_1, \tau_2)$ are annhilated by some powers of $\mathcal J$ and so are finite-dimensional.

Now, the finite-dimensionality of Ext-groups implies that 
$\widehat{\mathcal Z}\otimes_{\mathcal Z}\mathrm{Ext}^i_{\mathcal H}(\tau_1, \tau_2) \cong \mathrm{Ext}^i_{\mathcal H}(\tau_1, \tau_2)$. Similarly, $\widehat{\mathcal Z'}\otimes_{\mathcal Z'}\mathrm{Ext}^i_{\mathcal H'}(\omega \times \tau_1, \omega \times \tau_2) \cong \mathrm{Ext}^i_{\mathcal H'}(\omega \times \tau_1, \omega \times\tau_2)$. We now have that the theorem follows from Theorem \ref{theorem on extensions and parabolic induction} and Proposition \ref{prop ext identify}.
\end{proof}

\section{Notion of Strong Ext Relevance} \label{section -  definition of strong ext relevance}

In this Section, we recall the definition of strong $\Ext$ relevance of Arthur parameters as introduced in \cite{qa25}. The notion of strong $\Ext$ relevance involves the Aubert-Zelevinsky involution. Hence we first recall the definition of the Aubert-Zelevinsky involution.

\subsection{The Aubert-Zelevinsky Involution} 

Given an irreducible representation $\pi$ of $\GL_n(F)$ the Aubert-Zelevinsky involution (denoted by the symbol $D$) takes $\pi$ to another irreducible representation of $\GL_n(F)$. This involution can also be defined in the more general context of any reductive $p$-adic group. 

Let $G$ be a reductive $p$-adic group. Fix a minimal parabolic $P_0$ of $G$. Given a smooth irreducible representation $\pi$ of $G$, consider the virtual representation, \[ \widehat{D}_G(\pi) = \sum_{P} (-1)^{\mathrm{rank}(M)} \operatorname{Ind}_P^G \big( r_P (\pi)), \] where the above sum is taken over all parabolic subgroups $P=MN$ of $G$ containing $P_0$.

It is known from the work of Aubert \cite{au96} that there $\widehat{D}_G(\pi)$ is an irreducible representation of $G$ upto a sign, that is, there exists $\epsilon\in {\pm 1}$ such that $D(\pi) = \epsilon \widehat{D}_G(\pi)$ is an irreducible representation of $G$. We define $D(\pi)$ to be the Aubert-Zelevinsky dual of $\pi$. 

\subsubsection{Aubert-Zelevinsky Dual of a Speh Representation}

We recall the following result about the the Aubert-Zelevinsky dual of a Speh representation. 
\begin{theorem}\cite[Theorem B]{ta86} \label{Aubert-Zelevinsky Dual of a Speh Representation}
Let $a,b\in \mathbb{Z}_{\geq 0}$ and $\rho\in \Irr$ be a unitarizable cuspidal representation. Let $\pi=u_{\rho}(a,b)$ be a Speh representation of $\G_n$. Then, \[ D(u_{\rho}(a,b)) = u_{\rho}(b,a). \]   
\end{theorem}

\begin{remark}
The theorem admits a natural restatement in terms of Arthur parameters. Let $ u_{\rho}(a,b) $ be a Speh representation with Arthur parameter
$$\phi \otimes V_a \otimes V_b, $$
where $\phi$ is an irreducible representation of $ W_F $ attached to the cuspidal representation $\rho$. Then, the Aubert--Zelevinsky dual of $ u_{\rho}(a,b) $ is $ u_{\rho}(b,a) $, whose Arthur parameter is
$$\phi \otimes V_b \otimes V_a.$$
\end{remark}

\subsection{Strong Ext Relevance}
We recall the notion of \textit{strong} $\Ext$ relevance for pairs of Arthur parameters as in \cite{qa25}. We remark that this definition of strong $\Ext$ relevance is informed by the original GGP notion of relevance as well as a duality theorem due to Nori and Prasad \cite[Theorem 2]{np20}.

\begin{definition} \label{definition of strong ext relevance}
 Let $m, n \in \mathbb{Z}_{\geq 0}$ and let $\pi_1$ and $\pi_2$ be Arthur type representations of $\GL_m(F)$ and $\GL_n(F)$ respectively. Let $\A(\pi_1)$ and $\A(\pi_2)$ denote their respective Arthur parameters. We say that $\pi_1$ and $\pi_2$ are \textit{strong} $\Ext$ relevant if there exist admissible homomorphisms $\{\phi_i\}_{i=1}^{i=r+s}$ and $\{\psi_i\}_{i=1}^{i=t+u}$ of $W_F$ with bounded image and positive integers $a_1,a_2,\cdots a_r,$ $ b_{r+1},b_{r+2},\cdots b_{r+s},$ $ c_1,c_2,\cdots c_{r+s},$ $d_1,d_2,\cdots d_t, e_{t+1},e_{t+2},$ $\cdots e_{t+u},$ $f_1,f_2,\cdots f_{t+u}$ such that, \[ \A(\pi_1) = \sum_{i=1}^{r} \phi_i \otimes V_{c_i} \otimes V_{a_i} \oplus \sum_{i=r+1}^{r+s} \phi_i \otimes V_{c_i} \otimes V_{b_i - 1} \oplus \sum_{i=1}^{t} \psi_i \otimes V_{f_i} \otimes V_{d_i}  \oplus  \sum_{i=t+1}^{t+u} \psi_i \otimes V_{e_i-1} \otimes V_{f_i} \] and, \[ \A(\pi_2) = \sum_{i=1}^{r} \phi_i \otimes V_{c_i} \otimes V_{a_i-1}  \oplus\sum_{i=r+1}^{r+s} \phi_i \otimes V_{c_i} \otimes V_{b_i} \oplus \sum_{i=1}^{t} \psi_i \otimes V_{d_i-1} \otimes V_{f_i} \oplus  \sum_{i=t+1}^{t+u} \psi_i \otimes V_{f_i} \otimes V_{e_i}. \]   
\end{definition}

\begin{remark}
Note that if one were to restrict oneself to the first two terms in the above definition, one would precisely recover the original GGP notion of relevant $A$-parameters.   
\end{remark}

\begin{remark}
It is shown in \cite{qa25} that if $\pi_1$ and $\pi_2$ are products of tempered and anti-tempered (Aubert-Zelevinsky dual of tempered) representations then strong $\Ext$ relevance gives a necessary and sufficient condition for $\Ext$ branching.   
\end{remark}

\subsection{Reformulation in Representation Theoretic terms}
We also recall the representation-theoretic reformulation of the above definition of strong $\Ext$ relevance from \cite[Definition 1.2]{qa25}. Let $\pi$ be an irreducible representation of $\G_n(F)$. We denote its highest derivative by $\pi^{(h)}$, and for convenience, set $\pi^{-} = \nu^{1/2}\pi^{(h)}$. According to \cite[Theorem 14]{lm14}, the highest derivative of $\nu^{1/2}u_{\rho}(a,b)$ is $u_{\rho}(a,b-1)$, i.e., $u_{\rho}(a,b)^- = u_{\rho}(a,b-1).$ This leads to the following equivalent, representation-theoretic definition of strong $\Ext$ relevance.

\begin{definition}  \cite[Definition 1.2]{qa25}
 Let $m, n \in \mathbb{Z}_{\geq 0}$ and let $\pi_1$ and $\pi_2$ be Arthur type representations of $\GL_m(F)$ and $\GL_n(F)$ respectively. Then $\pi_1$ and $\pi_2$ are strong $\Ext$ relevant if there exist Speh representations, 
 \[ \pi_{m,1},\pi_{m,2},\cdots,\pi_{m,r},\pi_{n,1},\pi_{n,2},\cdots,\pi_{n,s}\] and, \[ \pi_{p,1},\pi_{p,2},\cdots,\pi_{p,t},\pi_{q,1},\pi_{q,2},\cdots,\pi_{q,u}\] such that, 
 \[ \pi_1 = \pi_{m,1}\times \cdots \pi_{m,r} \times \pi_{p,1}^- \times \cdots\times \pi_{p,t}^- \times \pi_{n,1}\times \cdots\times \pi_{n,s}  \times  D(\pi_{q,1}^-)\times \cdots\times D(\pi_{q,u}^-) \] and, \[ \pi_2 = \pi_{m,1}^-\times \cdots \pi_{m,r}^- \times \pi_{p,1} \times \cdots\times \pi_{p,t}  \times  D(\pi_{n,1}^-)\times \cdots\times D(\pi_{n,s}^-) \times \pi_{q,1} \times \cdots\times \pi_{q,u}. \]
\end{definition}

\section{Proof of Theorem \ref{theorem on non tempered ext branching}} \label{section - theorem on ext branching}

\subsection{Prelimnaries} We first recall and collect some preliminary lemmas that we shall need in the course of our proof of Theorem \ref{theorem on non tempered ext branching}.

\subsubsection{K\"unneth Formula}

We shall require the following K\"unneth formula in the course of our calculations of $\Ext$ modules. 
\begin{theorem}(\cite[Theorem 3.5]{pr24}) \label{kunneth formula}
Let $H_1$ and $H_2$ be $p$-adic groups and suppose that $H_1$ is reductive. Let $E_1$ and $F_1$ (resp. $E_2$ and $F_2$) be smooth representations of $H_1$ (resp. $H_2$). If $E_1$ and $F_1$ have finite lengths then, \[ \Ext^i_{H_1\times H_2} (E_1\otimes E_2,F_1\otimes F_2) = \sum_{k=0}^i \Ext^k_{H_1} (E_1,F_1) \otimes \Ext^{i-k}_{H_2} (E_2, F_2). \]
\end{theorem} 

\subsubsection{Duality Theorem} We have the following cohomological duality theorem (see \cite[Theorem 2]{np20}) due to Nori and Prasad.

\begin{theorem} \label{duality theorem of nori and prasad}
Let $\pi_1$ be a smooth irreducible representations of $\G_n$ and $\pi_2$ be any smooth representation of $\G_n$. Then for any integer $i\geq 0$, \[ \Ext^i_{\G_n}(\pi_1,\pi_2)^{\vee} \simeq \Ext^{d(\pi_1)-i}_{\G_n}(\pi_2,D(\pi_1)).\] Here $d(\pi_1)$ is the split rank of the Levi subgroup of $\G_n$ which carries the cuspidal support of $\pi_1$. 
\end{theorem}

\subsubsection{Transfer Lemma}

The next lemma provides a bridge between the branching problem for the pair $(\G_n, \G_{n-1})$ and that for $(\G_{n-1}, \G_{n-2})$, allowing us to pass information in either direction. This connection will play an important role in the proof of Theorem \ref{theorem on non tempered ext branching}.

\begin{lemma}\cite[Proposition 4.1]{ch22} \label{transfer lemma}
Let $\pi_1$ and $\pi_2$ be irreducible representations of $\G_n$ and $\G_{n-1}$ respectively. Then for any cuspidal representation $\sigma$ of $\G_2$ such that $\sigma \not \in \csuppline(\nu^{-1/2}\pi_1^{\vee})\cup \csuppline(\pi_2)$ we have that, \[ \Ext^i_{\G_{n-1}}(\pi_1,\pi_2^{\vee}) = \Ext^i_{\G_{n}}(\pi_2 \times \sigma,\pi_1^{\vee}) \] for all integers $i\geq 0$.  
\end{lemma}

\subsubsection{Main Reduction Lemma}

The following lemma is Lemma 4.3 from \cite{ch22}, where it is originally stated for $\Hom$ spaces, but applies equally to $\Ext$ modules, as noted in \cite[Lemma 6.1]{qa25}. This result reduces the $\Ext$ branching problem to a calculation of $\Ext$ modules over the same group. Although we do not recall the proof here, we remark that it is a consequence of a coarse version of the Bernstein-Zelevinsky filtration introduced in the work \cite{ch22}.

\begin{lemma} \label{lemma for the key reduction step}
Let $\pi_1$ and $\pi_2$ be Arthur type representations of $\G_n$ and $\G_{n-1}$ respectively. Suppose that \[ \pi_1= u_{\rho_1}(a_1,b_1)\times u_{\rho_2}(a_2,b_2)\times \cdots \times u_{\rho_r}(a_r,b_r)\] and \[ \pi_2 = u_{\tau_1}(c_1,d_1)\times u_{\tau_2}(c_2,d_2)\times \cdots \times u_{\tau_s}(c_s,d_s)\] where $\rho_i$ ($i=1,2,\cdots,r$) and $\tau_j$ ($j=1,2,\cdots,s$) are unitary cuspidal representations. Further suppose that $a_1+b_1 \geq a_i + b_i$ and $a_1+b_1 \geq c_j + d_j$  for all $i=1,2,\cdots,r$ and $j=1,2,\cdots,s$. Let $\sigma$ be a unitary cuspidal representation of $\G_{a_1n(\rho_1)}$ such that $\nu^{1/2}\sigma\not\in \csupp_{\mathbb{Z}}(\pi_2)$ and $\sigma\not\in \csupp_{\mathbb{Z}}(\pi_1)$. Then for all integers $i\geq 0$ we have, \[ \Ext^i_{\G_{n-1}}(\pi_1,\pi_2) \cong \Ext^i_{\G_{n-1}} (u_{\rho_1}(a_1,b_1-1)\times (\sigma \times \pi_{1}^{\prime})|_{\G_{a}},\pi_2)  \] where, $a=n - n(\rho_1)a_1(b_1-1)-1$ and $\pi_{1}^{\prime} = u_{\rho_2}(a_2,b_2)\times \cdots \times u_{\rho_r}(a_r,b_r)$.

\end{lemma}

\subsection{Proof of Theorem \ref{theorem on non tempered ext branching}}

We now give the proof of Theorem \ref{theorem on non tempered ext branching}. The initial part proceeds along the same lines as in \cite[Theorem 8.1]{qa25} and so we only sketch some of those steps referring to \cite{qa25} for the details.

\begin{theorem} 
Let $\pi_1$ and $\pi_2$ be irreducible Arthur type representations of $\G_n$ and $\G_{n-1}$ respectively. If $\pi_1$ and $\pi_2$ are strong $\Ext$ relevant then, \[ \Ext^i_{\G_{n-1}}(\pi_1,\pi_2)\neq 0 \] for some integer $i\geq 0$.
\end{theorem}

\begin{proof}

Suppose $\pi_1 = \pi_{1,1} \times \pi_{1,2} \times \cdots \times \pi_{1,r}$ and 
$\pi_2 = \pi_{2,1} \times \pi_{2,2} \times \cdots \times \pi_{2,s}$ where,
\begin{align*}
\pi_{1,i} &= u_{\rho_i}(a_i,b_i) \quad \text{for } i = 1,\ldots,r, \\
\pi_{2,j} &= u_{\tau_j}(c_j,d_j) \quad \text{for } j = 1,\ldots,s,
\end{align*} where the $\rho_i$ ($i=1,2,\cdots,r$) and $\tau_j$ ($j=1,2,\cdots,s$) are some cuspidal representations. 

Given that $\pi_1$ and $\pi_2$ are strong $\operatorname{Ext}$-relevant, we proceed by induction on $m(\pi_1,\pi_2)$ which denotes the number of non-cuspidal factors of $\pi_1$ and $\pi_2$. The base case is when $m(\pi_1,\pi_2)=0$, in which case both $\pi_1$ and $\pi_2$ are generic and by \cite[Theorem 3]{pr93} we conclude that $\Hom_{\G_{n-1}}(\pi_1,\pi_2)\neq 0$. We now assume that the theorem holds for all integers less than $m(\pi_1,\pi_2)$.

Given that $\pi_1$ and $\pi_2$ are irreducible, we can permute their factors to arrange that one of the following cases occurs. We can ensure that either:
\begin{enumerate}
    \item $a_1+b_1\geq a_i+b_i$ and $a_1+b_1\geq c_j+d_j$ for all $i$ and $j$.
    \item $c_1+d_1\geq c_j+d_j$ and $c_1+d_1\geq a_i+b_i$ for all $i$ and $j$.
\end{enumerate}

\textbf{Reduction Step :} By an application of the transfer lemma, Lemma \ref{transfer lemma} arguing similarly as in Case 2 of the proof of \cite[Theorem 7.1]{qa25} we can without loss of generality assume that the first case above occurs.

\textbf{Back to the proof:} Hence without loss of generality let us suppose that $a_1+b_1\geq a_i+b_i$ and $a_1+b_1\geq c_j+d_j$ for all $i$ and $j$. Choose a unitary cuspidal representation $\sigma$ of $\G_2$ such that $\nu^{1/2}\sigma\not \in \csuppline(\pi_2)$ and $\sigma\not \in \csuppline(\pi_1)$. Then by Lemma \ref{lemma for the key reduction step} we have that for all integers $i\geq 0$, \begin{equation} \label{equation one in proof of if direction}
\Ext^i_{\G_{n-1}}(\pi_1,\pi_2) = \Ext^i_{\G_{n-1}} (\pi_{1,1}^{-}\times (\sigma \times \pi_{1}^{\prime})|_{\G_{a}},\pi_2)  \end{equation} where, $a=n - n(\rho_1)a_1(b_1-1)-1$ and $\pi_{1}^{\prime} = u_{\rho_2}(a_2,b_2)\times \cdots \times u_{\rho_r}(a_r,b_r)$. Here $\pi_{1,1}^{-} = u_{\rho_1}(a_1,b_1)^- = u_{\rho_1}(a_1,b_1-1)$.

Now since the representations $\pi_1$ and $\pi_2$ are strong $\Ext$ relevant one of the following two cases must occur. \begin{enumerate}
    \item [(A)] $\pi_{2,m} = \pi_{1,1}^-$ for some $m=1,2,\cdots,s$. In this case $\pi_2 = \pi_{1,1}^- \times \pi_2^{\prime}$ where, $\pi_2^{\prime} = \pi_{2,1} \times\cdots \times \pi_{2,m-1} \times \pi_{2,m+1} \times \cdots \times \pi_{2,s}$. Moreover, the representations $\pi_{1}^{\prime}$ and $\pi_2^{\prime}$ are strong $\Ext$ relevant.
    \item [(B)] $\pi_{2,m} = D(\pi_{1,1}^-)$ for some $m=1,2,\cdots,s$. In this case $\pi_2 = D(\pi_{1,1}^-) \times \pi_2^{\prime\prime}$ where, $\pi_2^{\prime\prime} = \pi_{2,1} \times\cdots \times \pi_{2,m-1} \times \pi_{2,m+1} \times \cdots \times \pi_{2,s}$. Moreover, the representations $\pi_{1}^{\prime}$ and $\pi_2^{\prime\prime}$ are strong $\Ext$ relevant.
\end{enumerate}

\textbf{Case (A):} Let us suppose that Case (A) above occurs, that is, $\pi_2 = \pi_{1,1}^- \times \pi_2^{\prime}$, and the representations $\pi_{1}^{\prime}$ and $\pi_2^{\prime}$ are strong $\Ext$ relevant. Since $\pi_{1}^{\prime}$ and $\pi_2^{\prime}$ are strong $\Ext$ relevant we have that the representations $\sigma\times \pi_{1}^{\prime}$ and $\pi_2^{\prime}$ are strong $\Ext$ relevant. By the induction hypothesis we conclude that, 
\begin{align} \label{equation one point five in proof of if direction}
     \Ext^*_{\G_a}(\sigma\times \pi_{1}^{\prime},\pi_2^{\prime})\neq 0.
\end{align}

By \cite{as20} we know that $(\sigma \times \pi_{1}^{\prime})|_{\G_{a}}$ is locally finitely generated, that is, each Bernstein component of $(\sigma \times \pi_{1}^{\prime})|_{\G_{a}}$ is finitely generated. We denote $\Sigma = (\sigma\times \pi_{1}^{\prime})|_{\G_{a}}$. Let $\mathfrak s_0$ denote the inertial class of $\pi_2^{\prime}$. We see that only Bernstein component of the class $\mathfrak s_0$ contributes to the above $\Ext$ module with other components making no contribution. So we conclude that, \[ \Ext^*_{\G_a}(\Sigma_{\mathfrak s_0},\pi_2^{\prime})\neq 0. \]

Let $\mathcal J_0$ be the maximal ideal annihilating $\pi_2^{\prime}$ in the center $\mathcal Z$ corresponding to Hecke algebra $\mathcal H$ for the Bernstein component of the class $\mathfrak s_0$. By Corollary \ref{cor non-vanishing} on taking completion we have that, \begin{equation} \label{equation two in proof of if direction}
\Ext^*_{\widehat{\mathcal C}(\mathcal J_0)}(\widehat{\Sigma_{\mathfrak s_0}},\pi_2^{\prime})\neq 0.    
\end{equation}  

Returning back to Equation \ref{equation one in proof of if direction} by applying second adjointness, taking the opposite Jacquet module of $\pi_2 = \pi_{1,1}^- \times \pi_2^{\prime}$ suppose that we get a composition factor of the form $\Psi_1 \otimes \Psi_2 $. Then by applying the Kunneth formula we can check that it contributes to the RHS of Equation \ref{equation one in proof of if direction} only if $\csupp(\Psi_1)= \csupp (\pi_{1,1}^-)$ and $\csupp(\Psi_2)= \csupp (\pi_2^{\prime})$. Hence $\Ext^i_{\G_{n-1}} (\pi_{1,1}^{-}\times \Sigma_{\mathfrak {s}},\pi_{1,1}^- \times \pi_2^{\prime}) = 0$ for any inertial class $\mathfrak s$ not equal to $\mathfrak s_0$.

Hence have that, \[ \Ext^i_{\G_{n-1}}(\pi_1,\pi_2) = \oplus_{\mathfrak s} \Ext^i_{\G_{n-1}} (\pi_{1,1}^{-}\times \Sigma_{\mathfrak {s}},\pi_{1,1}^- \times \pi_2^{\prime}) = \Ext^i_{\G_{n-1}} (\pi_{1,1}^{-}\times \Sigma_{\mathfrak {s_0}},\pi_{1,1}^- \times \pi_2^{\prime}). \] Applying Corollary \ref{cor non-vanishing} once again, we then conclude by taking completions that,
\begin{equation} \label{equation three in proof of if direction}
\Ext^i_{\widehat{\mathcal C}(\mathcal J_0')}(\pi_1,\pi_2) = \Ext^i_{\widehat{\mathcal C}(\mathcal J_0')} (\pi_{1,1}^{-}\times \widehat{\Sigma_{\mathfrak {s_0}}},\pi_{1,1}^- \times \pi_2^{\prime}).    
\end{equation} Note here that we use the fact that the completion commutes with induction (see e.g. Lemma \ref{lemma on commutation of completion and induction}).

By Corollary \ref{corollary on nonzeroness 1}, Equation \ref{equation two in proof of if direction} and Equation \ref{equation three in proof of if direction} we conclude that \[ \Ext^i_{\widehat{\mathcal C}(\mathcal J_0')}(\pi_1,\pi_2)\neq 0 \] for some integer $i\geq 0$. 

\textbf{Case (B):} On the other hand suppose that Case (B) occurs. In this case, $\pi_2 = D(\pi_{1,1}^-) \times \pi_2^{\prime\prime}$, and the representations $\pi_{1}^{\prime}$ and $\pi_2^{\prime\prime}$ are strong $\Ext$ relevant. In this case we have that \begin{equation} \label{equation four in proof of if direction}
\Ext^i_{\G_{n-1}}(\pi_1,\pi_2) = \Ext^i_{\G_{n-1}} (\pi_{1,1}^{-}\times \Sigma_{\mathfrak {s_0}},D(\pi_{1,1}^-) \times \pi_2^{\prime}).    
\end{equation}

By the duality theorem of Nori and Prasad \cite{np20} to show that $\Ext^*_{\G_{n-1}}(\pi_1,\pi_2)\neq 0$ it is enough to show that, 
\begin{equation} \label{equation reuction in proof of if}
    \Ext^*_{\G_{n-1}} (\pi_{1,1}^- \times D(\pi_2^{\prime}), \pi_{1,1}^{-}\times \Sigma_{\mathfrak {s_0}})\neq 0. 
\end{equation} Applying the duality theorem of Nori and Prasad to equation \ref{equation one point five in proof of if direction} we have that, \[ \Ext^*_{\G_{a-1}} ( D(\pi_2^{\prime}),  \Sigma_{\mathfrak {s_0}}) \neq 0. \]
Note that $\mathrm{Ext}^{i'}_{\mathrm G_{a-1}}(\Sigma_{\mathfrak s_0}, \pi_2')$ is finite-dimensional for all $i'$ as $\Sigma_{\mathfrak s_0}$ is finitely-generated and $\pi_2'$ is of finite length as $\mathrm{G}_{a-1}$-representations. So the ones obtained by the duality theorem are also finite-dimensional. We let $\mathcal J_0$ be the maximal ideal as in Case (A). Now, by Corollary \ref{cor non-vanishing},
\begin{equation} \label{equation five in proof of if direction}
     \mathrm{Ext}^*_{\widehat{\mathcal C}(\mathcal J_0)}(D(\pi_2'), \widehat{\Sigma_{\mathfrak s_0}}) \neq 0 .
\end{equation} Combining the above two equations and Corollary \ref{corollary on nonzeroness 1} we conclude that \[ \Ext^*_{\widehat{\mathcal C}(\mathcal J_0')}(\pi_{1,1}^- \times D(\pi_2^{\prime}),  \pi_{1,1}^{-}\times \widehat{\Sigma_{\mathfrak {s_0}}})\neq 0. \]
Now one also has that the above Ext-groups are finite-dimensional. By Corollary \ref{cor non-vanishing} again, we obtain equation \ref{equation reuction in proof of if} as desired.
\end{proof}

\section{Appendix: Commutativity of Completion and Induction} \label{section - commutativity of completion and induction} Let $n_1,n_2$ be positive integers and set $n=n_1+n_2$. Let $\pi_1$ and $\pi_2$ be finitely generated modules of $\mathcal H_{n_1}$ and $\mathcal H_{n_2}$ respectively. Consider the induced $\mathcal H_n$ module defined as, \[ \pi = \mathcal{H}_n \otimes_{\mathcal H_{n_1} \otimes \mathcal H_{n_2}} (\pi_1 \otimes \pi_2). \] Let $\mathcal J_1$ and $\mathcal J_2$ be maximal ideals in the centers $\mathcal Z_1$ and $\mathcal Z_2$ of $\mathcal H_{n_1}$ and $\mathcal H_{n_2}$ respectively. We let $\mathcal{J}_0 = \langle \mathcal J_1 \otimes \mathcal Z_2,\mathcal Z_1 \otimes \mathcal J_2 \rangle$ denote the ideal generated by $\mathcal J_1$ and $\mathcal J_2$ in $\mathcal Z_1 \otimes \mathcal Z_2$. We know that $\mathcal Z_1 \otimes \mathcal Z_2$ contains the center $\mathcal Z$ of $\mathcal H_n$. We set $\mathcal J = \mathcal J_0 \cap \mathcal Z$, that is, $\mathcal J$ is the ideal generated by $\mathcal J_1$ and $\mathcal J_2$ inside $\mathcal Z$.

\begin{lemma} \label{lemma on commutation of completion and induction}
Let $\widehat{\pi_1}$ and $\widehat{\pi_2}$ denote the $\mathcal J_{1}$-adic and $\mathcal J_{2}$-adic completions of $\pi_1$ and $\pi_2$ respectively. Let $\widehat{\pi}$ be the $\mathcal J$-adic completion of the module $\pi$. As $\mathcal{H}_n$ modules we have the isomorphism, \[ \widehat{\pi} = \mathcal{H}_n \otimes_{\mathcal H_{n_1} \otimes \mathcal H_{n_2}} (\widehat{\pi_1} \otimes \widehat{\pi_2}). \]
\end{lemma}

\begin{proof}
This lemma is quite standard in the context of graded Hecke algebras. We provide some details for the convenience for the reader. Let $\mathcal A_{n_i}=\mathbb C[y_1^{\pm 1}, \ldots, y_{n_i}^{\pm 1}]$ for $i=1,2$, and let $\mathcal A_n=\mathbb C[y_1^{\pm 1}, \ldots, y_n^{\pm 1}]$. Let $\widehat{\mathcal A_{n_1}}, \widehat{\mathcal A_{n_2}}$ and $\widehat{\mathcal A_n}$ be the $\mathcal J_1$-adic, $\mathcal J_2$-adic and $\mathcal J$-adic completions of $\mathcal A_{n_1}$, $\mathcal A_{n_2}$ and $\mathcal A_n$ respectively. We similarly use the notations for $\widehat{\pi_1}, \widehat{\pi_2}, \widehat{\pi_1\otimes \pi_2}$.

It follows from \cite[7.5(a)]{lu89} (and the Chinese remainder theorem) for $\mathcal A_n$, we have 
\[   \widehat{\mathcal A_n} \cong \widehat{\mathcal A_{n_1}} \otimes \widehat{\mathcal A_{n_2}} .
\]
This accordingly gives that $\widehat{\pi_1\otimes \pi_2} \cong \widehat{\pi_1}\otimes \widehat{\pi_2}$. Now, as $\mathcal J$ is in $\mathcal Z$ and is also in the center of $\mathcal H_{n_1}\otimes \mathcal H_{n_2}$, $\mathcal J$ commutes with $\mathcal H_n$ action and so we also have
\[  \widehat{\pi} \cong \mathcal H_n \otimes_{\mathcal H_{n_1} \otimes \mathcal H_{n_2}} (\widehat{\pi_1 \otimes \pi_2}) .
\]
Combining above isomorphisms, we obtain the lemma.
\end{proof}


\end{document}